\documentclass[a4paper,11pt,reqno]{amsart}
\usepackage{amsmath}
\usepackage[T1]{fontenc}
\usepackage{amssymb}
\usepackage{amsthm,graphicx}
\usepackage{color}
\usepackage[lmargin=2.5 cm,rmargin=2.5 cm,tmargin=3.5cm,bmargin=2.5cm,
paper=a4paper]{geometry}

\newcommand{\Ab}{\mathbf A}

\newcommand{\kp}{\kappa}
\newcommand{\Om}{\Omega}

\newcommand{\Fb}{\mathbf F}

\newcommand{\R}{\mathbb R}
\newcommand{\C}{\mathbb C}
\newcommand{\gse}{\mathrm E_{\rm gs}(\kappa,H,\varepsilon)}

\DeclareMathOperator{\curl}{curl}

\newtheorem{thm}{Theorem}[section]
\newtheorem{prop}[thm]{Proposition}
\newtheorem{lem}[thm]{Lemma}

\newtheorem{proposition}[thm]{Proposition}

\theoremstyle{remark}

\newtheorem{rem}[thm]{Remark}

\newcommand{\eq}{\begin{equation}}
\newcommand{\eeq}{\end{equation}}

\def\curl{\text{\rm curl\,}}

\def\v{\vskip}

\def\a{\alpha}

\def\A{\bold A}
\def\C{\bold C}

\def\R{\mathbb  R}

\def\Z{\mathbb Z}
\def\0{\bold 0}

\def\a{\bold a }

\def\C{\mathbb C}

\def\jb{{\bf j}}

\numberwithin{equation}{section}

\title[Thin plane domains]
{Thin domain limit and counterexamples \\ to strong diamagnetism}
 
\author{Bernard Helffer}
\author{Ayman Kachmar}
\address[B. Helffer]{Laboratoire de Math\'ematiques Jean Leray, Universit\'e de Nantes}
\email{Bernard.Helffer@univ-nantes.fr  }
\address[A. Kachmar]{Department of Mathematics, Lebanese University, Nabatieh, Lebanon.}
\email{ayman.kashmar@gmail.com}
\date{\today}
\thanks{Mathematics Subject Classification (2010): 35B40, 35P15, 35Q56}
\begin{document}
\begin{abstract}
We study the magnetic Laplacian and the Ginzburg-Landau functional in a thin planar, smooth, tubular domain and with a uniform applied magnetic field. We provide counterexamples to strong diamagnetism, and as a consequence, we prove  that the transition from the superconducting to the normal state is non-monotone.  In some  non-linear regime, we determine the structure of the order parameter and compute the super-current along the boundary of the sample.  Our results are in agreement with what was observed in the Little-Parks experiment, for a thin cylindrical sample. 
\end{abstract}
\maketitle

\tableofcontents

\section{Introduction}

\subsection{The Ginzburg-Landau model in a non-simply connected domain}\text{ }

Let $\omega$ and $\Omega$ be two simply connected bounded open sets in $\R^2$ such that $\overline\omega\subset\Omega$. We assume also that the boundary of $\Omega$, $\partial\Omega$, is smooth of class $C^2$. The domain $\Omega\setminus\overline{\omega}$ is then a non-simply connected domain with the single hole $\omega$.

The main question addressed in this paper is the inspection of  the Ginzburg-Landau (GL) functional
\begin{equation}\label{eq:GL*-omega}
\mathcal E_{\omega}(\psi,\Ab)=\int_{\Omega\setminus\overline{\omega}} \left(|(\nabla-i H \Ab)\psi|^2-\kappa^2|\psi|^2+\frac{\kappa^2}2|\psi|^4\right)\,dx+
H^2\int_\Omega|\curl(\Ab-\Fb)|^2\,dx\,,
\end{equation}
where $\kappa>0$ is the GL parameter, $H>0$ the intensity of the applied magnetic field, and
\begin{equation}\label{eq:H-om}
(\psi,\Ab)\in\mathcal H_\omega=H^1(\Omega\setminus\overline{\omega}\,;\C)\times H^1_{\rm div}(\Omega\,;\R^2)\,.
\end{equation}
The space $H^1_{\rm div}(\Omega;\R^2)$ consists of all vector fields in $H^1(\Omega;\R^2)$ satisfying ${\rm div}\Ab=0$ in $\Omega$ and $\nu\cdot\Ab=0$ on $\partial\Omega$, where $\nu$ is the interior normal vector field of $\partial\Omega$. The vector field $\Fb$ is the unique vector field satisfying
\begin{equation}\label{eq:F}
\curl\Fb=1\quad{\rm and}\quad \Fb\in H^1_{\rm div}(\Omega;\R^2)\,.
\end{equation}
A configuration $(\psi,\Ab)\in\mathcal H_\omega$ is said to be a critical point of the GL functional if it is a weak solution\footnote{  The weak formulation of \eqref{eq:GLeq} is precisely given in \cite[(10.9a)-(10.9b)]{FH-b}.} of the corresponding Euler-Lagrange equations, named GL equations in this context, and read as follows
\begin{equation}\label{eq:GLeq}
\left\{\aligned
-&\big(\nabla-i H\Ab\big)^2\psi=\kp^2(1-|\psi|^2)\psi &{\rm in}\ \Omega\setminus\overline{\omega}\,,\\
-&\nabla^\bot  \big(\curl(\Ab-\Fb)\big)= \frac{1}{H}{\rm Im}\big(\overline{\psi}(\nabla-iH {\bf A})\psi\big) & {\rm in}\ \Omega\setminus\overline{\omega}\,,\\
&\nu\cdot(\nabla-iH {\bf A})\psi=0 & {\rm on}\ \partial \Om\cup\partial\omega \,,\\
&{\rm curl}({\bf A}-\Fb)=0 & {\rm on}\ \omega\,,
\endaligned
\right.
\end{equation}
where  the operator $\nabla^\bot=(-\partial_{x_2},\partial_{x_1})$ is the Hodge gradient. 
\subsection{Normal states}~\\
 A critical point $(\psi,\Ab)$ is said to be trivial (or a normal state)  if  $\psi\equiv0$. It is said to  be a minimizer if it minimizes the functional  in the variational space $\mathcal H_\omega$\,. So one introduces the critical field
\begin{equation}\label{eq:critical-field}
H_{c,\omega}:=\sup\{H>0~:~\exists\,(\psi,\Ab), ~\psi\not\equiv0, ~(\psi,\Ab)~{\rm satisfies~}~\eqref{eq:GLeq}\}\,.
\end{equation}
A result by Giorgi-Phillips ensures that this critical field is indeed finite.
The question of estimating the critical field is closely related to  the spectral analysis of the magnetic Laplacian in $L^2(\Omega\setminus\overline{\omega})$\,,
\begin{equation}\label{eq:Le}
\mathcal L_\omega^b=-(\nabla-ib \Fb)^2
\end{equation}
with (magnetic) Neumann boundary condition on $\partial\Omega\cup\partial\omega$. Here $b\in\R_+$ is a parameter measuring the strength of the magnetic field. The operator $\mathcal L_{\omega}^b$ is actually defined via the closed quadratic form 
\begin{equation}\label{eq:qf*}
H^1(\Omega\setminus\overline{\omega})\ni u\mapsto \mathcal Q_{\omega,b}(u)=\int_{\Omega\setminus\overline{\omega}}|(\nabla-ib\Fb)u|^2\,dx\,.
\end{equation}
We denote by $\lambda(\omega,b)$ the lowest eigenvalue of the operator $\mathcal L_\omega^b$, which is given by the min-max principle as follows
\begin{equation}\label{eq:lambda-omega-b}
\lambda(\omega,b)=\inf_{u\in H^1(\Omega\setminus\overline{\omega})\setminus\{0\}} \frac{Q_{\omega,b}(u)}{\quad \|u\|_{L^2(\Omega\setminus\overline{\omega})}^2}\,.
\end{equation}
The relation between the eigenvalue $\lambda(\omega,b)$ and the critical field is displayed via the following well known result:
\begin{prop}\label{prop:nt-min}
For all $\kappa,H>0$, if $\lambda(\omega,H)<\kappa^2$, then  every minimizer of the GL functional is non-trivial. Consequently  the GL equations in \eqref{eq:GLeq} admit a non-trivial solution.
\end{prop}
The proof of Proposition~\ref{prop:nt-min} simply follows by computing the GL energy $\mathcal E_\omega(tu,\Fb)$ with $t>0$ and $u$ a ground state of the operator $\mathcal L_\omega^H$.  The parameter $t$ can be selected  sufficiently small to ensure that $\mathcal E_\omega(tu,\Fb)<0=\mathcal E_\omega(0,\Fb)$, which in turn guarantees the existence of a non-trivial minimizer of the GL energy in \eqref{eq:GL*-omega}.\\

%{\clr Pourquoi parler de (1.4) dans l'\'enonc\'e plut\^ot que d'un minimiseur.\\}

\subsection{The thin domain}~\\
In the sequel, we will introduce a small parameter $\varepsilon>0$, and  choose the hole   $\omega$   in the following manner
\begin{equation}\label{eq:hole-ep}
\omega:=\omega_\varepsilon=\{x\in\Omega~:~{\rm dist}(x,\partial\Omega)>\varepsilon\}\,.
\end{equation}
We will refer to the parameter $\varepsilon$  as the `thickness' of our thin domain, $\Omega_\varepsilon$, defined as follows
\begin{equation}\label{eq:Om-e}
\Omega_\varepsilon:=\Omega\setminus\overline{\omega_\varepsilon}=\{x\in\Omega~:~{\rm dist}(x,\partial\Omega)<\varepsilon\}\,.
\end{equation}
We define the eigenvalue $\lambda(\omega,b)$ and the GL energy $\mathcal E_\omega$ as follows
\begin{equation}\label{eq:lambda-GL-ep}
\lambda(\varepsilon,b):=\lambda(\omega_\varepsilon,b)\quad{\rm and}\quad \mathcal E_\varepsilon(\cdot,\cdot)=\mathcal E_{\omega_\varepsilon}(\cdot,\cdot)\,.
\end{equation}
Also, we shorten the notation for the critical field, introduced in \eqref{eq:critical-field}, and write,
\begin{equation}\label{eq:critical-field-ep}
H_{c}(\varepsilon):=H_{c,\omega_\varepsilon}\,.
\end{equation}
A critical point, solving \eqref{eq:GLeq} for $\omega=\omega_\varepsilon$, will be denoted by $(\psi,\Ab)_{\kappa,H,\varepsilon}$, to emphasize the dependence on the parameters $\kappa,H$ and $\varepsilon$.

We can sharpen the statement in Proposition~\ref{prop:nt-min} when the thickness parameter $\varepsilon$ is `small'.
\begin{thm}\label{thm:hc3}
Given $\kappa>0$, there exists $\varepsilon_0>0$ such that, for all $\varepsilon\in(0,\varepsilon_0]$ and $H\geq0$, the following two statements are equivalent.
\begin{itemize}
\item[(A)] There exists a non-trivial critical point $(\psi,\Ab)_{\kappa,H,\varepsilon}$\,.
\item[(B)] $H$ satisfies 
$\lambda(\varepsilon,H)<\kappa^2$\,.
\end{itemize}
\end{thm}
\medskip

\subsection{The magnetic Laplacian}~\\
Armed with Theorem~\ref{thm:hc3}, when estimating the critical field $H_{c}(\varepsilon)$ in the small `thickness'  limit,  $\varepsilon\to0_+$,  we are led  to estimating the eigenvalue $\lambda(\varepsilon,b)$, of the magnetic Laplacian $\mathcal L_{\omega_\varepsilon}^b$. After doing that, we will find that $H_c(\varepsilon)$ is asymptotically inversely proportional to $\varepsilon$.

For later use, we introduce 
\begin{equation}\label{eq:L(omega)}
L=\frac{|\partial\Omega|}{2}\,,
\end{equation}
where $|\partial\Omega|$ denotes the length of the boundary $\partial\Omega$.\\
 Since the domain $\Omega_\varepsilon$ is non-simply connected, it is no surprise  that the eigenvalue $\lambda(\varepsilon,b)$ depends on the circulation of the magnetic field around the hole of the domain. So we introduce the following quantity,
\begin{equation}\label{eq:gamma0}
\gamma_0=\frac1{|\partial\Omega|}\int_\Omega\curl \Fb \,dx=\frac{|\Omega|}{|\partial\Omega|}\,.
\end{equation}
\subsection{Main results}~\\
Our main results, Theorems~\ref{rem:t-r} and \ref{thm:nt-r} below, display  the dependence of the eigenvalue $\lambda(\varepsilon,b)$ on the circulation $\gamma_0$, in the `thin domain limit', $\varepsilon\to0_+$.
\begin{thm}\label{rem:t-r}~
For every $N>0$, there exist positive constants $\varepsilon_0,d_0,\delta_0$ such that, for all $\varepsilon\in(0,\varepsilon_0]$, the following holds
\begin{itemize}
\item[(A)] If $\varepsilon b\geq d_0$, then $\lambda(\varepsilon,b)\geq N$\,;
\item[(B)] If $0\leq \varepsilon b\leq \delta_0$, then $\Big|\lambda(\varepsilon,b)-\left(\frac{\pi}{L}\right)^2 \inf\limits_{n\in\mathbb Z}\left|n+\frac{b L\gamma_0}{\pi}\right|^2\Big|\leq \frac1N$\,.
\end{itemize}
\end{thm}
In light of Theorem~\ref{rem:t-r},  we see that 
\begin{equation}\label{1.15a}
\lambda(\varepsilon,b)\underset{\substack{\varepsilon b\to+\infty\\\varepsilon\to0}}{\longrightarrow}+\infty
\end{equation}
   and  
   \begin{equation}\label{1.16a}
   \lambda(\varepsilon,b)\underset{\substack{\varepsilon b\to0\\
\varepsilon\to0}}{\sim} \left(\frac{\pi}{L}\right)^2 \inf\limits_{n\in\mathbb Z}\left|n+\frac{b L\gamma_0}{\pi}\right|^2\in[0,\frac{\pi^2}{4L^2}]\,.
\end{equation}
 So it remains to analyze the regime where $\varepsilon b\propto 1$, thereby bridging the two regimes appearing  in Theorem~\ref{rem:t-r} above.
 
Of particular interest  is the behavior of the eigenvalue $\lambda(\varepsilon,b_\varepsilon)$, where
\begin{equation}\label{eq:Be}
b_\varepsilon=\frac{a}{\varepsilon}+c
\end{equation}
with $c\in\R$ and $a>0$. The constant $c$ will have an `oscillatory' effect that will be discussed in Subsection \ref{rem:little-parks} below.

In the regime \eqref{eq:Be}, a central role is played by the following quantities
\begin{equation}\label{eq:betan}
\beta_n(c,a,\varepsilon)=\left |n+\frac{L}\pi\left(\gamma_0\left(\frac{a}{\varepsilon}+c\right) +\frac{a}{2}\right)\right| 
\mbox{ for }  n\in \mathbb Z\,,
\end{equation}
and their infimum over $\mathbb Z\,$:
\begin{equation}\label{eq:betan0}
\mathfrak i_0(c,a,\varepsilon):=\inf_{n\in\mathbb Z}\beta_n(c,a,\varepsilon)  \in[0,\frac12] \,.
\end{equation}
 The infimum  is attained for one or two minimizers in $\mathbb Z$. The minimizer is unique  when \break $\frac{L}\pi\left(\gamma_0\left(\frac{a}{\varepsilon}+c\right) +\frac{a}{2}\right)\not\in\frac12\mathbb Z$ and denoted by $n_0= n_0 (c ,a, \varepsilon  )$.   If $\frac{L}\pi\left(\gamma_0\left(\frac{a}{\varepsilon}+c\right) +\frac{a}{2}\right)\in\frac12\mathbb Z\,$, we have  two minimizers, $n_0$ and $n_0+1\,$.

\begin{thm}\label{thm:nt-r}
If $b_\varepsilon$ is defined by \eqref{eq:Be} for some given $c\in\R$ and $a>0$, then
$$\lambda(\varepsilon,b_\varepsilon)= \frac{a^2}{12}+\left(\frac{\pi \, \mathfrak i_0(c,a,\varepsilon)}{L} \right)^2+\mathcal O (\varepsilon)\quad{\rm as}~\varepsilon\to0_+\,,$$
where $\mathfrak i_0(c,a,\varepsilon)$ was introduced in \eqref{eq:betan0}.
\end{thm}

\subsection{Remarks}

\begin{enumerate}
\item The conclusion in Theorem~\ref{thm:nt-r} is formally  consistent with the one  in Theorem~\ref{rem:t-r}. Actually, for $a=0$ and $c=b\,$, we recover the regime (B)  in Theorem~\ref{rem:t-r}, while regime (A)  corresponds to  $a=+\infty\,$.  Results on the multiplicity of the eigenvalue $\lambda(\varepsilon,b_\varepsilon)$ are discussed in Subsection~\ref{sec:multip}.

\item {\bf Comparison with the large $\kappa$ regime.}\\
In the interesting paper \cite{FP}, Fournais and Persson-Sundqvist prove that for the disc geometry, $\Omega=D(0,R)$, there exists a thickness $\varepsilon_0$ and a value $\kappa_0$ for the GL parameter such that the transition to the normal state is not monotone. Our contribution goes beyond that, since for any \emph{geometry} $\Omega$ and for any \emph{value} of the GL parameter, we will prove that the transition to the normal state is not monotone for a certain thickness $\tilde\varepsilon$ constructed in Sec.~\ref{rem:little-parks} (see Proposition~\ref{prop:little-parks} and Remark~\ref{rem:little-parks*}). 
\item {\bf Oscillations for bounded fields.}\\
The interesting contributions by Berger-Rubinstein \cite{BR} and Rubinstein-Schatzman \cite{RSch} establish oscillations for bounded fields $H$ and particular values of the GL parameter. They study the convergence of the GL  functional  $\mathcal E_\varepsilon$ to an effective one-dimensional functional. Their results continue to hold for $H\ll \frac1\varepsilon$. That  can be easily checked by the arguments used  in this paper. One  significant difference of our results is that they hold in the regime of  large applied magnetic field and yield an estimate of the critical magnetic field. Also, our arguments differ from those in \cite{RSch} and are connected to the spectral theory of the magnetic Laplacian in a thin domain. 
\item {\bf Three dimensional rings.}\\
Shieh and Sternberg \cite{SS} study the GL functional in a three dimensional ring (i.e. a domain of the form $\{x\in\R^3,~{\rm dist}(x,\mathcal C)<\varepsilon\}\,$ where $\mathcal C$ is a simple closed and smooth curve) and for an applied magnetic field inversely proportional to $\varepsilon$. They identify a one dimensional limiting problem  in the frame work of the $\Gamma$-convergence and their limiting problem shows oscillations interpreted in terms of the critical temperature. Our contribution holds in a simpler geometry but it displays the oscillations for the full GL model and not only in the limit problem. 
\end{enumerate}

\subsection{Concentration of the GL minimizers}\text{ }

It is natural to study the minimization of the GL energy, $\mathcal E_\varepsilon$, for $H=\frac{a}{\varepsilon}+c$.  
We define the  ground state energy
\begin{equation}\label{eq:gseGL}
\gse = \inf\{\mathcal G_\varepsilon(\psi,\Ab)~:~(\psi,\Ab)\in\mathcal H_{\omega_\varepsilon}\}\,,
\end{equation}
where the space $\mathcal H_{\omega_\varepsilon}$ was introduced in \eqref{eq:H-om}.

\begin{thm}\label{thm:min}
Given $\kappa,a>0$ and $c\in\R$, then, for $H=\frac{a}{\varepsilon}+c$,  as $\varepsilon>0$ tends to $0\,$,
\begin{equation} \label{eq:1.21}
\gse=- \dfrac{\left(\kappa^2-\mathfrak e_0(c,a,\varepsilon)\right)_+^2}{2\kappa^2}\, |\Omega_\varepsilon| +\mathcal O(\varepsilon^2)\,, 
\end{equation}
where 
\begin{equation}\label{eq:eff-en*}
\mathfrak e_0(c,a,\varepsilon)=\frac{a^2}{12}+\left(\frac\pi L \mathfrak i_0(c,a,\varepsilon)\right)^2
\end{equation}
and $\mathfrak i_0(c,a,\varepsilon)$ is introduced in \eqref{eq:betan0}.

Moreover, if $(\psi,\Ab)_{\varepsilon,H,\kappa}$ is a minimizer of the GL functional, then
\begin{equation}\label{eq:1.23}
\int_{\Omega_\varepsilon}\left(\kappa|\psi|^2-\frac{\left(\kappa^2-\mathfrak e_0(c,a,\varepsilon)\right)_+}{\kappa}\right)^2\,dx=\mathcal O(\varepsilon^2)\,.
\end{equation}
\end{thm}
 We can estimate the \emph{circulation} of the supercurrent of a minimizing configuration provided for some $\delta \in (0,\frac 12)$,  the following  two separation conditions hold
\begin{equation}\label{eq:ass-sep}
{\rm (SC)}_{\delta}: \quad {\rm dist}\left(\frac{L}\pi\left(\gamma_0\left(\frac{a}{\varepsilon}+c\right) +\frac{a}{2}\right),\frac12\mathbb Z\right)\geq \delta\,,
\end{equation}
and
\begin{equation}\label{eq:ass-sep'}
{\rm (SC)}'_{\delta}: \quad \kappa^2-\mathfrak e_0(c,a,\varepsilon)\geq \delta\,.
\end{equation}
 Note that, by Theorem~\ref{thm:nt-r}, the condition (SC)$'_\delta$ in \eqref{eq:ass-sep'} yields that  $\lambda(\varepsilon,H)<\kappa^2$, for $H=\frac{a}\varepsilon+c$. Consequently,  Proposition~\ref{prop:nt-min} yields that the minimizing configurations of the GL functional are non-trivial, thereby confirming the presence of the superconducting phase. The condition (SC)$_\delta$ yields that \eqref{eq:betan} has a unique  minimizer $n_0$ which satisfies $ n_0=\mathcal O(\varepsilon^{-1})$. Note finally  that, if the constants $\kappa$ and $a$ satisfy the relation
 $$
 \delta_0(a,\kappa):=\kappa^2-\frac{a^2}{12}-\left(\frac{\pi }{2L} \right)^2>0\,,
 $$
 then 
  (SC)$'_\delta$ holds for all $\delta\in(0,\delta_0)$. \\

For a vector field $\mathbf u$, we introduce the circulation along $\partial\Omega$ as follows
$$\oint_{\partial\Omega}\mathbf u\cdot dx:=\frac1{|\partial\Omega|}\int_{\partial\Omega} \mathbf u\cdot \mathbf t\,ds\,,$$
where $ds$ indicates the arc-length measure along the boundary $\partial\Omega$, and
$\mathbf t$ is the unit tangent vector along $\partial\Omega$ oriented in the counter clock-wise direction.

\begin{thm}\label{thm:min*}
 Given $\kappa,a>0$, $\delta\in(0,\frac12)$  and $c\in\R$, there exists $\varepsilon_0>0$ such that, for $\varepsilon\in(0,\varepsilon_0]$ satisfying {\rm(SC)}$_\delta$ and  {\rm(SC)}$'_\delta$,  $H=\frac{a}{\varepsilon}+c$ and  $(\psi,\Ab)_{\varepsilon,H}$ minimizing  the GL functional, 
\begin{equation}\label{eq:th16}
\oint_{\partial\Omega} \jb \cdot dx =  \frac{\left(\kappa^2-\mathfrak e_0(c,a,\varepsilon)\right)_+}{\kappa^2}\, \frac{4\pi n_0}{|\partial\Omega|} +o(\varepsilon^{-1})\,.
\end{equation}
Here $n_0\in\Z$ is   the minimizer of  \eqref{eq:betan}  and $\jb :={\rm Re}(i\psi\overline{(\nabla-iH\Ab)\psi})$ is the super-current.
\end{thm}

The proof of Theorem~\ref{thm:min*} is given in Section~\ref{Sec:min*}, where we  establish an estimate  compatible with the following expected behavior of the minimizing order parameter (up to a gauge transformation)
\begin{equation}\label{eq:def-u0}
\psi (x) \sim \frac{\left(\kappa^2-\mathfrak e_0(c,a,\varepsilon)\right)_+}{\kappa^2} \,\exp\left(i\frac{2\pi n_0 s}{|\partial\Omega|}\right)\,,
\end{equation}
where $s$ is the tangential arc-length variable of $x$ on $\partial\Omega$. The convergence in \eqref{eq:def-u0} will be made precise in Section \ref{Sec:min*} later (see  \eqref{eq:app-psi3} and \eqref{eq:app-psi-bnd2} in Proposition~\ref{prop:app-psi}).   

Interestingly, this is reminiscent of the surface superconductivity regime in type~II superconductors (see \cite{CR} and the references therein). 

\subsection*{Notation} Given $p\in[1,+\infty]$ and an open set $U\subset\R^2$, we denote by $\|\cdot\|_{p,U}$ the usual norm in the space $L^p(U)$.

\section{Proof of Theorems~\ref{rem:t-r}~\&~\ref{thm:nt-r}}\label{s2}

For the considerations in Theorems~\ref{rem:t-r} and \ref{thm:nt-r}, we assume that $b=b_\varepsilon$ is a function of $\varepsilon$. We will deal with the three regimes:
$$\varepsilon b_\varepsilon\ll1\,,\quad \varepsilon b_\varepsilon\approx 1\,,\quad \varepsilon b_\varepsilon\gg 1\,.$$

\subsection{Boundary coordinates}\label{ss2.2}~\\
Recall the definition of the geometric constants $L$ and $\gamma_0$ in \eqref{eq:L(omega)} and \eqref{eq:gamma0} respectively.
Let $M:[-L,L)\to \partial\Omega$ be the arc-length parameterization of the boundary so that $\mathbf t:=M'(s)$ is the unit tangent vector of $\partial\Omega$ oriented counter-clockwise. Choose $\varepsilon_0\in(0,1)$  sufficiently small such that the transformation
\begin{equation}\label{eq:Phi0}
\Phi_0:(s,t)\in [-L,L)\times (0,\varepsilon_0]\mapsto M(s)+t\,\nu(s)\in\Omega_{\varepsilon_0}
\end{equation}
is bijective, where $\nu(s)$ is the unit interior normal vector of $\partial\Omega$ at the point $M(s)$. 

In the sequel, suppose that $\varepsilon\in(0,\varepsilon_0]$. 
We denote by
\begin{equation}\label{eq:f}
f(s,t)=-\gamma_0-t+\frac{t^2}2k(s)\,,
\end{equation}
where $k(s)$ is the curvature of $\partial\Omega$ at $M(s)$, and $\gamma_0$ is the circulation of the applied magnetic field, introduced in \eqref{eq:gamma0}.

We have (see \cite[Lem.~F.1.1]{FH-b}):
\begin{equation}\label{eq:qf}
\int_{\Omega_\varepsilon}|(\nabla-i b_\varepsilon \Fb)u|^2\,dx=Q^{L,\varepsilon}_{b_\varepsilon}(v):=\int_{-L}^{L}\int_0^\varepsilon \Big(|\partial_t v|^2+\a^{-2}|(\partial_s-ib_\varepsilon f) \,v|^2\Big)\a\,dtds\,,
\end{equation}
where
\begin{equation}\label{eq:a}
\a (s,t)=1-tk(s)\,,
\end{equation}
\begin{equation}\label{eq:v}
v(s,t)=e^{ib_\varepsilon\varphi_0(s,t)}u(\Phi_0(s,t))\,,
\end{equation}
and  $\varphi_0(s,t)$ is a smoth function, $2L$-periodic with respect to the $s$-variable, and depends only on the vector field $\Fb$ and the geometry of the domain $\Omega$. Hence it is  independent  from $\varepsilon$ and the choice of the function $u$.  In fact we can take (see \cite[Eq.~(F.11)]{FH-b})
$$\varphi_0(s,t)=\int_0^t\tilde \Fb_2(s,t')dt'+\int_0^s\tilde\Fb_1(s',0)ds'-s\gamma_0\,,$$ 
where (see \cite[Eq.~(F.2)]{FH-b})
$$\tilde\Fb_1(s,t)=\a(s,t) \Fb\big(\Phi_0(s,t)\big)\cdot M'(s) \quad{\rm and}\quad \tilde \Fb_2(s,t)=\Fb\big(\Phi_0(s,t)\big)\cdot\nu(s)=0\,,$$
since $\Fb\in H^1_{\rm div}(\Omega)$.\\
Moreover, we can express the $L^2$-norm of $u$ in the following manner:
\begin{equation}\label{eq:norm-u}
\int_{\Omega_\varepsilon}|u(x)|^2\,dx=\int_{-L}^L\int_0^\varepsilon |v(s,t) |^2\, \a (s,t)\,dtds\,.
\end{equation}
~
\subsection{Reduction of the operator}\label{ss2.3}~\\
Let us assume now that $\varepsilon b_\varepsilon\leq M_0\,$, for some constant $M_0>0$. This hypothesis will be valid when for example \eqref{eq:Be} holds, or when we consider the conclusion (B)  in Theorem~\ref{rem:t-r}\,.

We can estimate  the quadratic form and the $L^2$ norm of $v$ as follows. There exist two constants $K>0$ and $\tilde\varepsilon_0\in(0,1)$,  depending  on the domain $\Omega$ only, such that, for all $\varepsilon\in(0,\tilde\varepsilon_0]$,
\begin{equation}\label{eq:Q=q}
\big(1-K\varepsilon\big)q^{L,\varepsilon}_{b_\varepsilon}(v)\leq
Q^{L,\varepsilon}_{b_\varepsilon}(v)\leq \big(1+K\varepsilon \big)q^{L,\varepsilon}_{b_\varepsilon}(v)
\end{equation}
 and  
\begin{equation}\label{eq:norm-s,t}
(1-K\varepsilon)\int_{-L}^{L}\int_0^\varepsilon |v|^2\,dtds\leq \|u\|^2_{2,\Omega_\varepsilon}\leq (1+K\varepsilon)\int_{-L}^{L}\int_0^\varepsilon |v|^2\,dtds
\,,
\end{equation}
where
\begin{equation}\label{eq:q-b-ep-L}
q^{L,\varepsilon}_{b_\varepsilon}(v)=\int_{-L}^{L}\int_0^\varepsilon 
\Big(|\partial_t v|^2+|(\partial_s-ib_\varepsilon f_0)v|^2\Big)dtds\,,\end{equation}
and
\begin{equation}\label{eq:f0}
f_0(t)=-\gamma_0-t\,.
\end{equation}
Actually, this follows from the following two estimates:
$$\big| {\bf a} (s,t)-1|\leq \|\kappa\|_{\infty}\varepsilon\quad{\rm and}\quad |f(t)-f_0(t)|\leq \frac12\|\kappa\|_\infty^2\varepsilon^2\,.$$
Let us introduce the  eigenvalue  $\hat\lambda(\varepsilon,b_\varepsilon)$ as follows
\begin{equation}\label{eq:hat-lambda}
\hat\lambda(\varepsilon,b_\varepsilon)=\inf_{v\in H^1([-L,L)\times(0,\varepsilon))\setminus\{0\}}\frac{q_{b_\varepsilon}(v)}{\quad \quad  \|v\|^2_{L^2([-L,L)\times(0,\varepsilon))}}\,.
\end{equation}
By the min-max principle,  we deduce the existence of $\tilde K$ and $\varepsilon_0 >0$ such that, for all $\varepsilon \in (0,\varepsilon_0]$, 
\begin{equation}\label{eq:hat-lambda=lambda}
\big|\lambda(\varepsilon,b_\varepsilon)-\hat\lambda(\varepsilon,b_\varepsilon)\big|\leq \tilde K \, \varepsilon\, \hat\lambda(\varepsilon,b_\varepsilon)\,.
\end{equation}

\subsection{Spectral analysis of the reduced operator}\label{ss2.4}

\subsubsection{Fourier modes}\label{sss2.4.1}

We decompose in Fourier modes to obtain the family of quadratic forms
\begin{equation} \label{eq:2.fam}q^\varepsilon_{n,b_\varepsilon}(v_n)=\int_0^\varepsilon \Big(|\partial_t v_n|^2+|(n L^{-1}\pi-b_\varepsilon f_0)v_n|^2\Big)dt\,.
\end{equation}
So we introduce for $\eta,b>0$,
$$\tilde q^{\varepsilon}_{\eta,b}(w)=\int_0^\varepsilon \Big(|\partial_t w|^2+|(\eta+b\, t) w|^2\Big)dt\,,$$
along with the corresponding eigenvalue
\begin{equation}\label{eq:lambda-fiber}
\lambda(\eta,\varepsilon,b_\varepsilon)=\inf_{\|w\|_{2,(0,\varepsilon)}^2=1} \tilde q^{\varepsilon}_{\eta,b}(w)\,.
\end{equation}
Note that
$$q^\varepsilon_{n,b_\varepsilon}(v_n)=\tilde q^\varepsilon_{\eta,b_\varepsilon}(v_n)$$
for 
\begin{equation}\label{eq:eta-n-ep}
\eta=\eta(n,b_\varepsilon) =nL^{-1}\pi+b_\varepsilon\gamma_0\,.
\end{equation}
 The eigenvalue in \eqref{eq:hat-lambda} can be expressed using the eigenvalues of the fiber operators as follows,
\begin{equation}\label{eq:2.13}
\hat\lambda(\varepsilon,b_\varepsilon)=\inf_{n\in\Z}\lambda\big(\eta(n,b_\varepsilon),\varepsilon,b_\varepsilon\big)\,.
\end{equation}

\subsubsection{Scaling}

Now, we assume that $b_\varepsilon$ satisfies \eqref{eq:Be} for some constants $a >0$ and $c\in\R$.  We do the change of variable $\tau= a \varepsilon^{-1} t$ and get
 \begin{equation}\label{eq:lambda-scaling}
 \lambda(\eta,\varepsilon,b_\varepsilon)=a^2\varepsilon^{-2}\mu(\alpha,\delta_\varepsilon,a,\zeta_\varepsilon ) \,,
 \end{equation}
 where $ \mu(\alpha,\delta,a,\zeta)$  is the lowest eigenvalue in $L^2(0, a)$ of the operator defined via the closed quadratic form, with $\delta \geq 0$ and $\zeta\in\R$,
\begin{equation}\label{eq:qf-h}
H^1(0,a)\ni u\mapsto h^{\alpha,a,\zeta}_{\delta}(u)=\int_0^{a}\Big(|\partial_\tau u|^2+\delta\, |(\alpha+\tau+ \zeta \tau)u|^2\Big) d\tau\,.
\end{equation}
The formula in \eqref{eq:lambda-scaling} is valid for
$\eta$ defined by \eqref{eq:eta-n-ep}, $\alpha=\alpha_n$, $\delta=\delta_\varepsilon$ and $\zeta =\zeta_\varepsilon $, where
\begin{equation}\label{eq:zeta-n}
\alpha_n=\frac{n\pi}{L}+\left(\frac{a}{\varepsilon}+c\right)\gamma_0\,,
\end{equation}
\begin{equation}\label{eq:c-epsilon}
\delta_\varepsilon=a^{-2}\varepsilon^{2}\quad{\rm and}\quad \zeta_\varepsilon =\frac{c\, \varepsilon}{a}\,.
\end{equation}

\subsubsection{The non-trivial regime} 

\subsubsection*{\bf Comparison with the 1D-Neumann Laplacian.}~\\
Let $\mathcal L_0$ be the $1D$ Neumann Laplace operator defined in $L^2(0,a)$ as follows
\begin{equation}\label{eq:op-L0}
\mathcal D_0\ni u\mapsto \mathcal L_0\, u=-\frac{d^2}{d\tau^2}\quad{\rm where~}\mathcal D_0=\{u\in H^2(0,a)~:~u'(0)=u'(a)=0\}\,.
\end{equation}
The min-max principle yields the following comparison for the eigenvalues defined via the quadratic form in \eqref{eq:qf-h} and those of the operator $\mathcal L_0$:
\begin{equation}\label{eq:mu-n=mu-n0}
\mu_n(\alpha,\delta,a,\zeta)\geq \mu_n(\mathcal L_0)\,.
\end{equation}
It is easy to check that
\begin{equation}\label{eq:mun(L0)}
\forall\,n\in\mathbb N \setminus \{0\},\quad \mu_n(\mathcal L_0)=\left(\frac{(n-1)\pi}{a}\right)^2 \,,
\end{equation}
hence the comparison in \eqref{eq:mu-n=mu-n0} is not effective for the first eigenvalue $\mu(\alpha,\delta,a,\zeta)$, since \break  $\mu_1(\mathcal L_0)=0$, however, for the second eigenvalue $\mu_2(\alpha,\delta,a,\zeta)$ we obtain
\begin{equation}\label{eq:mu2-h}
\mu_2(\alpha,\delta,a,\zeta)\geq \left(\frac\pi{a}\right)^2\,.
\end{equation}~

\subsubsection*{\bf Behavior of $\mu(\alpha,\delta,a,\zeta_\varepsilon )$ as $\delta\to0\,$.}~\\

We recall from \eqref{eq:c-epsilon} that $\lim\limits_{\varepsilon \rightarrow 0} \zeta_\varepsilon  =0\,$. Fix  positive constants $\zeta_0$ and   $ A$. 
We will first write an estimate of the eigenvalue $ \mu(\alpha,\delta,a,\zeta)$ that holds uniformly with respect to $ \alpha\in[-A,A]$ and $\zeta\in[- \zeta_0,\zeta_0]\,.$
 A standard argument of perturbation in $\delta$ allows us to  expand the eigenvalue $\mu(\alpha,\delta,a,\zeta)$ as follows
$$\mu(\alpha,\delta,a,\zeta)\underset{\delta\to0}=\mu_0+\delta\mu_1 +\mathcal O(\delta^2)\,.$$
We recall the proof for the commodity of the reader. We introduce a quasi-mode of the form
$$u:=u_0+\delta u_1\,,$$
so that
\begin{equation}
\left(-\frac{d^2}{d\tau^2}+\delta(\tau+\alpha+\zeta \tau)^2\right)(u_0+\delta u_1)= (\mu_0+\delta\mu_1)(u_0+\delta u_1 )  +\mathcal O(\delta^2)\,.
\end{equation}
Then the natural choice of
$\mu_0,\mu_1,u_0,u_1$  (depending smoothly on $\zeta$) would be 
\begin{align*}
&-\frac{d^2}{d\tau^2} u_0=\mu_0 \, u_0\,,\\
&\left(-\frac{d^2}{d\tau^2}-\mu_0\right)u_1+\Big((\tau+\alpha +\zeta \tau)^2-\mu_1(\zeta)\Big)u_0=0\,.
\end{align*}
We choose $\mu_0=\mu_1(\mathcal L_0)=0$ and $u_0=1$, in accordance with  \eqref{eq:mun(L0)}. In order to solve the equation for $u_1(\cdot,\zeta)$, we choose $\mu_1=\mu_1(\zeta)$ so that $\Big((\tau+\alpha+\zeta \tau )^2-\mu_1\Big)u_0$ is orthogonal to $u_0$ in $L^2(0,a)$, thereby obtaining  the  Feynman-Hellman formula,
\begin{equation}\label{defmu1}
\mu_1(\zeta)=\frac1a\int_0^a (\tau+\alpha+ \zeta \tau)^2\,d\tau=  \alpha^2 + \alpha (1+\zeta)  a + \frac 13 a^2 (1+ \zeta)^2 \,.
\end{equation}
 We note for later use that 
\begin{equation}\label{propmu1}
\mu_1(\zeta) \leq \left( (1+\zeta) a + |\alpha|\right)^2\,.
\end{equation}

With this choice, we solve the differential equation satisfied by $u_1$, with the boundary conditions $u_1'(0)=u_1'(a)=0$, and get, imposing that $u_1$ is orthogonal to $u_0$, a unique solution $u_1(\cdot,\zeta)$.

Now, the quasi-mode $u(\cdot,\zeta)=u_0(\cdot,\zeta)  + \delta\, u_1(\cdot,\zeta)$ satisfies $$u'(0,\zeta)=u'(a,\zeta)=0\,,$$
and
$$\Big\|\Big(-\frac{d^2}{d\tau ^2}+\delta(\alpha+\tau+ \zeta  \tau)^2- \delta \mu_1(\zeta)\Big)  u (\cdot,\zeta) \Big\|_{L^2(0,a)}\leq C_{A,a, \zeta_0}\, \delta ^2 \, \|u\|^2_{L^2(0,a)}\,,$$
which is valid for $\delta\in(0,\delta_{A,a,\zeta_0})$, where $\delta_{A,a,\zeta_0}$ and $C_{A,a,\zeta_0}$ are constants that depend only on $A$, $a$ and $\zeta_0\,$.

 Taking $\zeta=\zeta_\varepsilon $, 
the spectral theorem and the lower bound in \eqref{eq:mu2-h} then yield that there exist $\varepsilon_{A,a}$ and $\widehat C_{A,a}$ such that
\begin{equation}\label{eq:mu-delta=0}
\Big|\mu(\alpha,\delta,a,\zeta_\varepsilon ) - \delta \mu_1(0)\Big|\leq \widehat C_{M,a}\, (\delta+|\zeta_\varepsilon |)\,\delta\,.
\end{equation}
This motivates us to introduce the following quantity
\begin{equation}\label{eq:m(alpha,delta)}
m(\alpha,a):=\mu_1(0)= \frac{a^2}{12}+\Big(\alpha+\frac{a}{2}\Big)^2\,.
\end{equation}

\begin{rem}\label{rem:mult-mu(alpha)}
Combining \eqref{eq:mu2-h} and \eqref{eq:mu-delta=0}, we see that, if $|\alpha| \leq A\,$,  there exists $\delta_0>0$ such that, for $\delta,\varepsilon\in(0,\delta_0)$, the eigenvalue $\mu(\alpha,\delta,a,\zeta_\varepsilon )$ is simple. 
\end{rem}
~

\subsubsection*{\bf Minimization of  $\mu(\alpha,\delta,a,\zeta_\varepsilon )$.}~\\

We are interested in estimating the quantity
\begin{equation}\label{eq:inf-alpha}
\mu_0(\delta,a,\zeta_\varepsilon ):=\inf_{\alpha\in\mathcal J_\varepsilon} \mu(\alpha,\delta,a,\zeta_\varepsilon )
\end{equation}
where
\begin{equation}\label{eq:index-J}
\mathcal J_\varepsilon=\{ \alpha_n=\frac{n\pi}{L}+\gamma_0\frac{a}{\varepsilon}+\gamma_0c~:~n\in\mathbb Z\}\,.
\end{equation}
Choose $n_0=n_0(\varepsilon)\in\mathbb Z$ so that
${|\alpha_{n_0}|}=\inf\limits_{n\in\mathbb Z}|\alpha_n|$. 
%{\clr Juste avant $|\alpha_{n_0}| = =\inf\limits_{n\in\mathbb Z}|\alpha_n|$.} 
Clearly, $$\alpha_{n_0}\in[-\frac{\pi}{2L},\frac{\pi}{2L}]\,.$$

Using the constant function $u\equiv 1$ as a test function in the quadratic form in \eqref{eq:qf-h}, we get the existence of $\varepsilon_0 >0$ such that,  for all  $\delta>0$ and $\varepsilon\in(0,\varepsilon_0]$,  
\begin{equation}
\mu(\alpha_{n_0},\delta,a,\zeta_\varepsilon )
\leq  \delta \mu_1(\zeta_\varepsilon) \leq 2\delta\left(a+\frac{\pi}{2L}\right)^2\,.
\end{equation}
Here we have used for the last inequality that $\zeta_\varepsilon$ tends to $0$ and $\alpha=\alpha_{n_0}$ in \eqref{propmu1}.\\

Noticing that, for $|\alpha|\geq 10(a+\frac{\pi}{2L})$,  $$\inf\limits_{0\leq \tau\leq a}(\alpha+\tau+\zeta_\varepsilon\tau)^2\geq \big(9 (a+\frac{\pi}{2L}) -|\zeta_\varepsilon|a\big)^2 \geq 4(a+\frac\pi{2L})^2\,,$$
for $\varepsilon$ sufficiently small, 
 we get by the min-max principle 
\begin{equation}\label{eq:lb-alpha-big}
\mu(\alpha,\delta,a,\zeta_\varepsilon )\geq 4   \left(a+\frac\pi{2L}\right)^2\delta >\mu(\alpha_{n_0},\delta,a,\zeta_\varepsilon )\geq \mu_0(\delta,a,\zeta_\varepsilon )\,.\end{equation}

This proves that the minimization in \eqref{eq:inf-alpha} can be restricted to  $\alpha\in [-A,A]\cap\mathcal J_\varepsilon$ with \break  $A=10(a+\frac{\pi}{2L})$. In light of \eqref{eq:mu-delta=0}, it is enough to minimize the function in \eqref{eq:m(alpha,delta)} with respect to $\alpha$.  Therefore,  there exist $\delta_0=\delta_0(a,c,\gamma_0,L)>0$ and $C_0=C_0(a,c,\gamma_0,L)>0\,$, such that, for all $\delta,\varepsilon\in(0,\delta_0)\,$,
\begin{equation}\label{eq:inf-alpha*}
\Big|\mu_0(\delta,a,\zeta_\varepsilon )- \Big(\frac{a^2}{12}+\inf_{n\in\mathbb Z} \Big|\frac{n\pi}{L}+\frac{a}{\varepsilon}\gamma_0+c\gamma_0+\frac{a}{2}\Big|^2\Big)\delta\Big|\leq C_0\, (\delta+\varepsilon)\delta\,.
\end{equation}
%{\clr All the constants depend also on $L$ and $\gamma_0$ ?}

\subsection{End of the proofs}

\subsubsection{The regime $b_\varepsilon\propto \frac1\varepsilon$\,.}\label{sec:int-regime}~

Collecting \eqref{eq:hat-lambda=lambda}, \eqref{eq:2.13}, \eqref{eq:lambda-scaling} and \eqref{eq:inf-alpha*}   (with $\delta=\delta_\varepsilon$ defined in \eqref{eq:c-epsilon}), we get, as $\varepsilon$ tends to $0$,  with  $b_\varepsilon=\frac{a}{\varepsilon}+c\,$,  the asymptotics stated in Theorem \ref{thm:nt-r}.\\

\subsubsection{The regime $\varepsilon b_\varepsilon\ll 1$\,.}~

In this case, we restart from Subsections \ref{ss2.3} and \ref{ss2.4}.
We choose $n_0(\varepsilon)\in\mathbb Z$ so that 
$$\left |n_0(\varepsilon) +\frac{b_\varepsilon L\gamma_0}{\pi}\right|=\inf_{n\in\mathbb Z}\left | n+\frac{b_\varepsilon L\gamma_0}{\pi}\right|\,,$$
and set 
$$\beta_{n,\varepsilon} := n +\frac{b_\varepsilon L\gamma_0}{\pi}\,.$$
Clearly, $$\beta_{n_0(\varepsilon),\varepsilon}\in[-\frac12,\frac12]\,.
$$
Using the function $u(s,t)=e^{in_0(\varepsilon)\pi s/L}$ as a test function, we get by a straightforward computation
\begin{equation}\label{eq:ub-intr}
\hat\lambda(\varepsilon,b_\varepsilon)\leq \left(\frac{\pi}{L}\right)^2\beta_{n_0(\varepsilon),\varepsilon}^2+\mathcal O(\varepsilon b_\varepsilon)\,.
\end{equation}
For the reverse inequality, we decompose in Fourier modes and do the rescaling $\tau=\varepsilon^{-1}t$, to get the following quadratic form,
$$\varepsilon^{-2}\int_0^1\left(|\partial_\tau u|^2+\varepsilon^2 \left|\left(\frac{\pi}{L}\beta_{n,\varepsilon}+\varepsilon b_\varepsilon\tau\right)u\right|^2\right)\,d\tau\geq \varepsilon^{-2}\int_0^1\varepsilon^2\left((1-\varepsilon b_\varepsilon)\left|\frac{\pi}{L}\beta_{n,\varepsilon}\right|^2-\varepsilon b_\varepsilon\right)|u|^2\,d\tau\,.$$
So, we get by the min-max principle that
$$
 \hat\lambda(\varepsilon, b_\varepsilon)\geq \inf_{n\in\mathbb Z}
\left((1-\varepsilon b_\varepsilon)\left|\frac{\pi}{L}\beta_{n,\varepsilon}\right|^2-\varepsilon b_\varepsilon\right)=\left(\frac{\pi}{L}\right)^2\beta_{n_0(\varepsilon),\varepsilon}^2+\mathcal O(\varepsilon b_\varepsilon)\,.$$
Finally, we use \eqref{eq:hat-lambda=lambda} to conclude the estimate for $\lambda(\varepsilon,b_\varepsilon)$   (Statement (B) in Theorem \ref{rem:t-r}).\\

\subsubsection{The regime $\varepsilon b_\varepsilon\gg 1$\,.}~

In this situation, we can not use the estimate in \eqref{eq:Q=q}, since replacing $b_\varepsilon f$ by $b_\varepsilon f_0$ produces a large error (see \eqref{eq:f} and \eqref{eq:f0}). 

We  rescale the variables as follows, $t=\varepsilon\tau$ and $ s=\varepsilon^{-2}b_\varepsilon^{-1}\sigma$. We obtain two constants $k>0$ and $\varepsilon_0\in(0,1)$ such that, for all $\varepsilon\in(0,\varepsilon_0]$,
$$\lambda(\varepsilon,b_\varepsilon)\geq (1-k\varepsilon)\tilde\lambda(\varepsilon,b_\varepsilon)\,,$$
where 
$$\tilde\lambda(\varepsilon,b_\varepsilon)=\inf_{\|v\|_{L^2(\mathbb T_\varepsilon)}=1} \tilde Q_{\varepsilon}(v)\,.$$
Here, 
\begin{align*}
&\mathbb T_\varepsilon=[-L_\varepsilon,L_\varepsilon)\times(0,1),
\\
&L_\varepsilon=\varepsilon^{-2}b_\varepsilon^{-1}L\,,\\
&\tilde Q_{\varepsilon}(v)= \int_{\mathbb T_\varepsilon}\Big(\varepsilon^{-2}|\partial_\tau v|^2+\varepsilon^2b_\varepsilon^2|(\partial_\sigma-if_\varepsilon)v|^2\Big)\,d\sigma\,d\tau\,,\end{align*}
and
$$f_\varepsilon(\sigma,\tau)=\varepsilon^{-1}\gamma_0-\tau+\frac{\varepsilon\tau^2}2\kappa(\varepsilon^{-2}b_\varepsilon^{-1}\sigma)\,.$$
We now prove that $\tilde\lambda(\varepsilon,b_\varepsilon)\underset{\varepsilon\to0}{\longrightarrow}+\infty\,$.\\
 Note that
$$\tilde Q_{\varepsilon}(v)\geq\min(\varepsilon^{-2},\varepsilon^2b_\varepsilon^2)\int_{\mathbb T_\varepsilon}\Big(|\partial_\tau   v |^2+|(\partial_\sigma-if_\varepsilon) v|^2\Big)\,d\sigma\,d\tau\,,$$
and
$$ \left|f_\varepsilon(\sigma,\tau)-f_\varepsilon^0(\tau)\right|\leq 2\|\kappa\|_\infty\varepsilon\quad{\rm where}~f_\varepsilon^0(\tau)=\varepsilon^{-1}\gamma_0-\tau\,.$$
Consequently,
\begin{align*}
 \int_{\mathbb T_\varepsilon}\Big(|\partial_\tau { v}|^2+|(\partial_\sigma-if_\varepsilon){ v}|^2\Big)\,d\sigma\,d\tau&\geq \int_{\mathbb T_\varepsilon}\Big(|\partial_\tau {  v}|^2+ \frac12|(\partial_\sigma-if_\varepsilon^0){ v}|^2-8 \|\kappa\|_\infty^2 \varepsilon^2|v|^2\Big)\,d\sigma\,d\tau\\
&\geq \Big( \frac 12 e(\varepsilon)  -8 \|\kappa\|_\infty^2 \varepsilon^2\Big)\int_{\mathbb T_\varepsilon}|v|^2\,d\sigma d\tau\,,\end{align*}
where
$$e(\varepsilon)=\inf_{\|v\|_{L^2(\mathbb T_\varepsilon)=1}}\int_{\mathbb T_\varepsilon} |\partial_\tau v|^2+ |(\partial_\sigma-i f_\varepsilon^0(\tau))v|^2\,d\sigma d\tau\,.$$
The min-max principle now  yields
$$
\tilde\lambda(\varepsilon,b_\varepsilon)\geq \frac12\Big(\min(\varepsilon^{-2},\varepsilon^2b_\varepsilon^2)\Big) \Big(e(\varepsilon)  -  16 \|\kappa\|_\infty^2 \varepsilon^2 \Big)\,.$$
By decomposition into Fourier modes, we may show that
$$e(\varepsilon) \geq \inf_{\alpha\in\R}\mu(\alpha,1,1,0) \,,$$ 
where $\mu(\alpha,\delta,a,0)$ is the eigenvalue defined  via the quadratic form in \eqref{eq:qf-h}, for $\delta=1$,  $a=1$ and $\zeta =0$.   

Using the min-max principle, it is easy to check that the function  $\alpha \mapsto \mu(\alpha,1,1,0)$ is continuous, positive-valued, and tends to $+\infty$ as $|\alpha| \rightarrow +\infty$. Consequently, it attains its minimum, i.e.
there exists $\alpha_0\in\R$ such that $$ \inf\limits_{\alpha\in\R}\mu(\alpha,1,1,0)=\mu(\alpha_0,1,1,0)>0\,.$$
 This proves that $\liminf\limits_{\varepsilon\to0_+}e(\varepsilon)>0$ and finishes the proof of $\tilde\lambda(\varepsilon,b_\varepsilon)\to+\infty$ in this regime  (Statement (A) in Theorem~\ref{rem:t-r}).

\section{Proof of Theorems~\ref{thm:hc3}~and~\ref{thm:min}}

\subsection{A priori estimates}\text{ }

There exists $\varepsilon_0\in(0,1)$ such that, for all $\varepsilon\in(0,\varepsilon_0]$, $\kappa,H>0$, every critical point $(\psi,\Ab)$  satisfies \cite[Ch.~10]{FH-b}
\begin{equation}\label{eq:apriori}
\begin{aligned}
&\|\psi\|_{\infty,\Omega_\varepsilon}\leq 1\,,\\
&\|(\nabla-iH\Ab)\psi\|_{2,\Omega_\varepsilon}\leq \kappa\, \|\psi\|_{2,\Omega_\varepsilon}\,,\\
&\|\nabla(\curl\Ab-\Fb)\|_{2,\Omega}\leq \frac1H\|(\nabla-iH\Ab)\psi\|_{2,\Omega_\varepsilon}\|\psi\|_{2,\Omega_\varepsilon}\,.
\end{aligned}
\end{equation}
%
%Now, assume that  $H\geq \sigma_0$, where $\sigma_0>0$ is a  given constant.\\
Noting that the first eigenvalue of the one dimensional  Dirichlet Laplacian $-\frac{d^2}{dt^2}$ in $L^2(0,\varepsilon)$  equals $(\frac{\pi}{\varepsilon})^2$, we get from \eqref{eq:apriori},  observing that $\curl(\Ab-\Fb)$ satisfies the Dirichlet condition on $\partial\Omega_\varepsilon$,
\begin{equation}\label{eq:A-F1}
\begin{aligned}
\|\curl(\Ab-\Fb)\|_{2,\Omega }&=\|\curl(\Ab-\Fb)\|_{2,\Omega_\varepsilon }\\
& \leq \frac{C \varepsilon}{\pi}\, \|\nabla\curl(\Ab-\Fb)\|_{2,\Omega_\varepsilon}\\ 
&  \leq \frac{C\varepsilon}{\pi H}\|(\nabla-iH\Ab)\psi\|_{2,\Omega_\varepsilon}\|\psi\|_{2,\Omega_\varepsilon}\\
&\leq  \frac{C \kappa \varepsilon}{\pi  H}\, \|\psi\|_{2,\Omega_\varepsilon}^2  \\
& =\mathcal O(\varepsilon H^{-1})\|\psi\|_{2,\Omega_\varepsilon}^2\,.
\end{aligned}
\end{equation}
Consequently, by the div-curl inequality in $\Omega$
\begin{equation}\label{eq:A-Fa}
\|\Ab-\Fb\|_{H^1(\Omega)}\leq  \tilde C \, \|\curl(\Ab-\Fb)\|_{2,\Omega}=\mathcal O(\varepsilon H^{-1})\|\psi\|_{2,\Omega_\varepsilon}^2\,.
\end{equation}
By the  embedding of $H^1(\Omega)$ in $L^p(\Omega)$ for $p\in[2,+\infty)$, we find ,  using the first line of \eqref{eq:apriori},
\begin{equation}\label{eq:A-Fp}
\|\Ab-\Fb\|_{p,\Omega}=\mathcal O(\varepsilon H^{-1})\|\psi\|_{2,\Omega_\varepsilon}^2=\mathcal O(\varepsilon^2  H^{-1})\,.
\end{equation}
We write by Cauchy's inequality,  
\begin{equation}\label{eq:lb-ke}
\|(\nabla-i H\Ab)\psi\|_{2,\Omega_\varepsilon}^2\geq (1-\eta)\|(\nabla-i H\Fb)\psi\|_{2,\Omega_\varepsilon}^2-\eta^{-1}H^2\|(\Ab-\Fb)\psi\|_{2,\Omega_\varepsilon}^2\,,
\end{equation}
where $\eta\in(0,1)$ and $(\psi,\Ab)_{\kappa,H,\varepsilon}$ is  a critical configuration.

We estimate the term $\|(\Ab-\Fb)\psi\|_{2,\Omega_\varepsilon}^2$ using H\"older's inequality and \eqref{eq:A-Fp} as follows
\begin{equation}\label{eq:A-F}
\|(\Ab-\Fb)\psi\|_{2,\Omega_\varepsilon}^2\leq \|\Ab-\Fb\|_{4,\Omega_\varepsilon}^2\|\psi\|_{4,\Omega_\varepsilon}^2=\mathcal O(\varepsilon^2  H^{-2})\|\psi\|_{2,\Omega_\varepsilon}^4\|\psi\|_{4,\Omega_\varepsilon}^2\,.
\end{equation}
Again, 
H\"older's inequality yields  $$ 
\|\psi\|_{2,\Omega_\varepsilon}^2\leq |\Omega_\varepsilon|^{1/2}\|\psi\|_{4,\Omega_\varepsilon}^2=\mathcal O(\varepsilon^{1/2}) \|\psi\|_{4,\Omega_\varepsilon}^2\,.
$$
 Thus, from \eqref{eq:lb-ke} and \eqref{eq:A-F}, we get the following lower bound,
\begin{equation}\label{eq:lb-ke*}
\|(\nabla-i H\Ab)\psi\|_{2,\Omega_\varepsilon}^2
\geq (1-\eta)\|(\nabla-i H\Fb)\psi\|_{2,\Omega_\varepsilon}^2-C\eta^{-1}\varepsilon^3\|\psi\|_{4,\Omega_\varepsilon}^6\,,
\end{equation}
where $C>0$ is a constant independent from $\eta$ and $H$.

Using this estimate, we can bound the GL functional from below as follows:
\begin{equation}\label{eq:GL-lb}
0\geq \mathcal E_\varepsilon(\psi,\Ab)\geq (1-\eta)\mathcal E_{\varepsilon}(\psi,\Fb)-\eta\kappa^2\|\psi\|_{2,\Omega_\varepsilon}^2-C\eta^{-1}\varepsilon^3\|\psi\|_{4,\Omega_\varepsilon}^6\,,
\end{equation}
and this is true for any critical configuration $(\psi,\Ab)_{\kappa,H,\varepsilon}\,$. \\
\subsection{Proof of Theorem~\ref{thm:hc3}}~

 Having Proposition \ref{prop:nt-min} in mind, we have only to prove that (A) implies (B).\medskip

\paragraph{{\bf Step 1:} First restriction}~\medskip

Using the constant function as a quasi-mode, we get that, for all $\varepsilon,H>0\,$,  
\begin{equation}\label{uppbz}
 \lambda(\varepsilon, H)\leq H^2\|\Fb\|_\infty^2\,,
 \end{equation}
where we take the $L^\infty$-norm on $\Omega$ in order to get the uniformity in $\varepsilon$.\\
 Thus, if $H <\sigma_0(\kappa)$ with
 $$
 \sigma_0(\kappa):= \kappa / \|\Fb\|_\infty\,,
 $$
 we have $\lambda(\varepsilon,H) < \kappa^2$ and (B) is satisfied.
 
 {\bf   From now on, we consider $H\geq \sigma_0(\kappa)$ and prove  that (A) implies (B) under this additional condition.}\medskip

\paragraph{{\bf Step 2:} Second restriction}~\medskip

We assume that (A) holds. 
Since $ \lambda(\varepsilon,H)\to+\infty$ as $\varepsilon H\to+\infty$  and $\varepsilon \to 0\,$ (see \eqref{1.15a}), we find $\Lambda_0$  and $\varepsilon_0 >0$ such that, for $\varepsilon H\geq \Lambda_0$  and $0< \varepsilon \leq \varepsilon_0$, $ \lambda(\varepsilon, H)>2\kappa^2$.\medskip

 The lower bound in \eqref{eq:GL-lb} used with $\eta=\varepsilon$, and the min-max principle, yield that, 
\begin{align*}
0&\geq (1-\varepsilon)\big(\lambda(\varepsilon,H)-\kappa^2\big)\|\psi\|_{2,\Omega_\varepsilon}^2-\varepsilon\kappa^2\|\psi\|_{2,\Omega_\varepsilon}^2-C\varepsilon^2\|\psi\|_{4,\Omega_\varepsilon}^6\\
&\geq (1-2\varepsilon)\kappa^2\|\psi\|_{2,\Omega_\varepsilon}^2-C\varepsilon^2\|\psi\|_{4,\Omega_\varepsilon}^6\,.
\end{align*}
Noting that,  because $|\psi|\leq 1$, 
$$\|\psi\|_{4,\Omega_\varepsilon}^6=\left( \int_{\Omega_\varepsilon} |\psi|^4dx\right)^{\frac 32} \leq \|\psi\|_{2,\Omega_\varepsilon}^3 \leq |\Omega_\epsilon|^\frac 12\,  \|\psi\|_{2,\Omega_\varepsilon}^2\,,$$ 
we get, for some positive constants $C_\kappa$ and $\varepsilon_0(\kappa)$, 
$$0\geq (1-2\varepsilon-C_\kappa \varepsilon^2)\kappa^2\|\psi\|_{2,\Omega_\varepsilon}^2\,,$$   for $ \varepsilon \in (0, \varepsilon_0(\kappa)] $ and any $\psi$ corresponding to a  critical  configuration.\\
This proves the existence of  a positive $\varepsilon_1(\kappa)$ such  that  $\psi\equiv0$ when $\varepsilon \in (0,\varepsilon_1(\kappa)]$ in contradiction with (A).

Hence at this stage, we have proven the existence of $\Lambda_0$ and $\varepsilon_1$ such that if  (A) holds then  $H\leq  \Lambda_0\varepsilon^{-1}$ 
for $\varepsilon \in (0,\varepsilon_1]$.

\paragraph{{\bf Step 3:} Proof in the remaining  case}~

We assume that (A) holds and that $ 0<\sigma_0(\kappa) \leq  H\leq \Lambda_0\varepsilon^{-1}$. There exist $\varepsilon_0$ and  $\Lambda$ such that, for all $\varepsilon\in (0,\varepsilon_0]$, 
\begin{equation}\label{uppb3}
\lambda(\varepsilon, H)\leq \Lambda\,.
\end{equation}
 This simply follows after
combining \eqref{eq:hat-lambda=lambda} and \eqref{eq:ub-intr}.  

%{\clr Le th\'eor\`eme ne dit pas exactement cela. Il faut se replonger dans sa preuve !\\}
We  introduce
\begin{equation}\label{defDelta}
\Delta=\kappa^2\|\psi\|_{2,\Omega_\varepsilon}^2-\|(\nabla-i  H\Ab)\psi\|_{2,\Omega_{\varepsilon}}^2=\kappa^2\|\psi\|_{4,\Omega_\varepsilon}^4\,.
\end{equation}
The hypothesis on the non-triviality of $\psi$ ensures that $\Delta>0$. 
 Also, as a consequence of  the first inequality in \eqref{eq:apriori}, we get
 \begin{equation}\label{eq:0<D<1}  
 0 < \Delta \leq \kappa^2 |\Omega_\varepsilon| =\mathcal O (\varepsilon) \,.
 \end{equation}
 Notice that the H\"older inequality yields that 
 \begin{equation}\label{uppb2}
 \kappa^2\|\psi\|_{2,\Omega_\varepsilon}^2\leq  \kappa^2  |\Omega_\varepsilon|^\frac 12 \, \|\psi\|_{4,\Omega_\varepsilon}^2\leq C\sqrt{\varepsilon}\, \Delta^{1/2}\,.
 \end{equation}
By \eqref{eq:lb-ke*} and the min-max principle, we write, for any $\eta\in(0,1)$,
\begin{equation}\label{eq:3.6}
\|(\nabla-i H\Ab)\psi\|_{2,\Omega_\varepsilon}^2\geq (1-\eta)\lambda(\varepsilon, H)\|\psi\|_{2,\Omega_\varepsilon}^2-C\eta^{-1}\varepsilon^3\|\psi\|_{4,\Omega_\varepsilon}^6\,,
\end{equation}
and we infer the following lower bound,
\begin{equation}\label{eq:3.6a} -\Delta\geq (\lambda(\varepsilon, H)-\kappa^2)\|\psi\|_{2,\Omega_\varepsilon}^2-\eta\lambda(\varepsilon, H)\|\psi\|_{2,\Omega_\varepsilon}^2-C\eta^{-1}\varepsilon^{3}\|\psi\|_{4,\Omega_\varepsilon}^6\,.
\end{equation}
Using \eqref{uppb3},  \eqref{defDelta}, \eqref{eq:0<D<1}, and \eqref{uppb2},  we get, from \eqref{eq:3.6a} with  $\eta= \varepsilon\Delta^{1/2}$ (note that $\eta\in(0,1)$ by \eqref{eq:0<D<1} for $\varepsilon$ small enough),
$$- (1- \hat C \varepsilon^\frac 32 ) \Delta\geq (\lambda(\varepsilon, H)-\kappa^2)\|\psi\|_{2,\Omega_\varepsilon}^2\,.$$
But $\Delta>0$, by our hypothesis, hence this yields for $\varepsilon$ small enough that
$$(\lambda(\varepsilon, H)-\kappa^2)\|\psi\|_{2,\Omega_\varepsilon}^2<0 \,,$$
which implies (B) after observing that $\|\psi\|_{2,\Omega_\varepsilon}\neq 0\,$.~\medskip

\subsection{Proof of Theorem~\ref{thm:min}}~

Let $(\psi,\Ab)_{\kappa,H,\varepsilon}$ be a minimizing configuration for $H=\frac{a}{\varepsilon}+c$. We start with the inequality in \eqref{eq:GL-lb} with $\eta=\varepsilon$.  Since $|\psi|\leq 1$ everywhere, \eqref{eq:GL-lb}   yields, for some constant $C>0$, 
\begin{equation}\label{eq:GL-lb**}
\mathcal E_{\varepsilon}(\psi,\Ab)\geq (1-\varepsilon)\mathcal E_\varepsilon(\psi,\Fb)- C \varepsilon^2\,.
\end{equation}
The quadratic form part in $\mathcal E_\varepsilon(\psi,\Fb)$ can be bounded from below by the min-max principle and Theorem~\ref{thm:nt-r}, so that 
\begin{equation}
\begin{aligned}
\mathcal E_\varepsilon(\psi,\Fb)
&\geq \left(\mathfrak e_0(c,a,\varepsilon)-\kappa^2+\mathcal O(\varepsilon)\right)\|\psi\|_{2,\Omega_\varepsilon}^2+\frac{\kappa^2}{2}\|\psi\|_{4,\Omega_\varepsilon}^4\\
& \geq - \left(\kappa^2 - \mathfrak e_0(c,a,\varepsilon)- \mathcal O(\varepsilon)\right)_+\|\psi\|_{2,\Omega_\varepsilon}^2+\frac{\kappa^2}{2}\|\psi\|_{4,\Omega_\varepsilon}^4 =:\mathcal R \,.
\end{aligned}
\end{equation}
We rewrite $\mathcal R$  in the form
\begin{equation*}
\mathcal R= \frac12\int_{\Omega_\varepsilon}\left(\kappa|\psi|^2-\frac{\left(\kappa^2-\mathfrak e_0(c,a,\varepsilon)-\mathcal O(\varepsilon)\right)_+}{\kappa}\right)^2\,dx-\frac{\left(\kappa^2-\mathfrak e_0(c,a,\varepsilon)-\mathcal O(\varepsilon)\right)_+^2}{2\kappa^2}|\Omega_\varepsilon|\,,
\end{equation*}
and get 
\begin{equation}\label{uppb5} 
\mathcal E_\varepsilon(\psi,\Fb) \geq \mathcal R \geq - \frac{\left(\kappa^2-\mathfrak e_0(c,a,\varepsilon)\right)_+^2}{2\kappa^2}|\Omega_\varepsilon|-\mathcal O(\varepsilon)|\Omega_\varepsilon|\,.
\end{equation}
After inserting this lower bound into \eqref{eq:GL-lb**}, we get the lower bound part in Theorem~\ref{thm:min}. \\
To obtain the matching upper bound, we write
$$\gse\leq \mathcal E_{\varepsilon}(u,\Fb)\,,$$
 and choose  as function $u(x)=\tilde u(s(x),t(x))$, which is  defined in the $(s,t)$ coordinates by
$$\tilde u(s,t)= 
v(s) \exp\Big(-iH\varphi_0(s,t)\Big)\,,\quad v(s)=\frac{\left(\kappa^2-\mathfrak e_0(c,a,\varepsilon)\right)_+^{1/2}}{\kappa}\exp\left(i\frac{n_0\pi s}{L}\right)\,.$$
Here $\varphi_0$ is the smooth function introduced in \eqref{eq:v} and $n_0\in\mathbb Z$ is defined just after \eqref{eq:betan0}. Collecting \eqref{eq:qf}, \eqref{eq:Q=q} and \eqref{eq:norm-s,t}, with the choice $b_\varepsilon=H$, we get
\begin{align*}
\mathcal E_\varepsilon(u,\Fb)&\leq \big(1+\mathcal O(\varepsilon)\big)\int_{-L}^L \int_0^\varepsilon\left( |(\partial_s -iHf_0)v|^2-\kappa^2|v|^2+\frac{\kappa^2}{2}|v|^4\right)\,dx+\mathcal O(\varepsilon)\kappa^2\int_{-L}^L \int_0^\varepsilon|v|^2dtds\\
&\leq  \big(1+\mathcal O(\varepsilon)\big)\,(2L\varepsilon)\,
\frac{\left(\kappa^2-\mathfrak e_0(c,a,\varepsilon)\right)_+^{2}}{2\kappa^2}
+\mathcal O(\varepsilon^2)\,.
\end{align*} 

The last statement in Theorem~\ref{thm:min} follows immediately of the upper bound, and the more accurate lower bound  of $\mathcal E_\varepsilon(\psi,\Fb) $:
\begin{equation}\label{uppb6} 
\begin{aligned}
\mathcal E_\varepsilon(\psi,\Fb) + \frac{\left(\kappa^2-\mathfrak e_0(c,a,\varepsilon)\right)_+^2}{2\kappa^2}|\Omega_\varepsilon| & 
\geq \mathcal R+  \frac{\left(\kappa^2-\mathfrak e_0(c,a,\varepsilon)\right)_+^2}{2\kappa^2} |\Omega_\varepsilon|\\ &   \geq  
 \frac12 \int_{\Omega_\varepsilon}\left( \kappa|\psi|^2-\frac{\left(\kappa^2  -\mathfrak e_0(c,a,\varepsilon)\right)_+} {\kappa} \right)^2  \, dx  - C \varepsilon |\Omega_\varepsilon| \,,
 \end{aligned}
\end{equation}
together with \eqref{eq:GL-lb**}.

\section{Analysis of  ground states and  strong diamagnetism -- Applications}

We discuss in this section some consequences that we obtain from the statement of Theorem~\ref{thm:nt-r} or along its proof.

\subsection{On the multiplicity of the eigenvalue $\lambda(\varepsilon,b_\varepsilon)$}\label{sec:multip}~

Along the proof of Theorem~\ref{thm:nt-r}, we get some information regarding the multiplicity of the eigenvalue $\lambda(\varepsilon,b_\varepsilon)$ when \eqref{eq:Be} holds. Interestingly, we get that $\lambda(\varepsilon,b_\varepsilon)$ is simple  when the `separation' condition (SC)$_\delta$ is satisfied.

\begin{proposition}\label{PropA}
 For any $\delta\in(0,\frac12)$, there exists $\varepsilon_0>0$ such that, for all $\varepsilon\in(0,\varepsilon_0]$ satisfying the   separation condition {\rm(SC)}$_\delta$ (see \eqref{eq:ass-sep}) the eigenvalue $\lambda(\varepsilon,b_{\varepsilon})$ is  simple (where $b_{\varepsilon}$ is given by \eqref{eq:Be}).
\end{proposition}

Proposition~\ref{PropA} can not be used  for the sequence $\left(\varepsilon_n= a \left(\frac1{2\gamma_0}\left(\frac{\pi n}{L}-a \right)-c\right)^{-1}\right)_{n\geq 1}$ since for any $\delta >0$  the values $\varepsilon_n$  violate the separation condition (SC)$_\delta$ for $n$ large enough. Proposition~\ref{PropB} addresses this degenerate situation, but unfortunately, it does not provide the exact value of the multiplicity. 

\begin{proposition}\label{PropB}
There exists $\varepsilon_0>0$ such that, for all $\varepsilon\in(0,\varepsilon_0]$, the multiplicity of $\lambda(\varepsilon,b_{\varepsilon})$ is $\leq 2$. 
 \end{proposition}\medskip

\begin{proof}[\bf Proof of Proposition \ref{PropA}]~

From Theorem~\ref{thm:nt-r}, we can choose $\varepsilon_0,M>0$ such that, for all $\varepsilon\in(0,\varepsilon_0]$, the eigenvalue $\lambda(\varepsilon,b_\varepsilon)$ satisfies
\begin{equation}\label{eq:ub-lambda-q}
\lambda(\varepsilon,b_\varepsilon)\leq \frac{a^2}{12}+\left(\frac{\pi}{L}\mathfrak i_0(c,a,\varepsilon)\right)^2+M\varepsilon\,,
\end{equation} 
where $b_\varepsilon=\frac{a}{\varepsilon}+c\,$.

Let us denote by $\mathfrak H_{n,\varepsilon}$ the self-adjoint operator defined by the quadratic form in \eqref{eq:qf-h} for $\alpha=\alpha_n$, $\delta=\delta_\varepsilon$ and $\zeta=\zeta_\varepsilon$ given in \eqref{eq:zeta-n} and \eqref{eq:c-epsilon}. We also denote by $\big(\mu_k(\mathfrak H_{n,\varepsilon})\big)_{k\geq 1}$  the  non decreasing   sequence of  eigenvalues of $\mathfrak H_{n,\varepsilon}$ counting multiplicities. Note that, for all $k\geq 1$, the eigenvalue $\mu_k(\mathfrak H_{n,\varepsilon})$ is simple, and  by \eqref{eq:mu2-h},
\begin{equation}\label{eq:mu-2-d-sum}
\mu_2(\mathfrak H_{n,\varepsilon})\geq \left(\frac\pi a\right)^2\,.\end{equation}
Now, using \eqref{eq:Q=q}-\eqref{eq:q-b-ep-L}, the min-max principle and the decomposition into Fourier modes (see \eqref{eq:2.fam}, \eqref{eq:lambda-scaling} and \eqref{eq:qf-h}), we get that, 
\begin{equation}\label{eq:d-sum}
 \lambda_k\left(\mathcal L_{\omega_\varepsilon}^{b_\varepsilon}\right)\geq 
  \delta_\varepsilon^{-1}(1-\tilde K\varepsilon)\lambda_k\left(\bigoplus_{n\in\Z}\mathfrak H_{n,\varepsilon}\right)\,,
\end{equation}
where $\tilde K>0$ is a constant,  and for an operator $\mathfrak P$, $\lambda_k(\mathfrak P)$ denotes the $k$'th  min-max eigenvalue of $\mathfrak P$.

As a consequence of \eqref{eq:d-sum},
\begin{multline}\label{eq:s-gs-details}
N\left(\mathcal L_{\omega_\varepsilon}^{b_\varepsilon},\frac{a^2}{12}+\left(\frac{\pi}{L}\mathfrak i_0(c,a,\varepsilon)\right)^2+M\varepsilon\right)\\
\leq {\rm Card}\left(\Big\{(n,k)\in\mathbb Z\times \mathbb N^*~:~\mu_k(\mathfrak H_{n,\varepsilon})\leq  \delta_\varepsilon(1-\tilde K\varepsilon)^{-1}\Big(\frac{a^2}{12}+\left(\frac{\pi}{L}\mathfrak i_0(c,a,\varepsilon)\right)^2 +M\varepsilon\Big)\Big\}\right)\,,
\end{multline}
where   $N(\mathcal L_{\omega_\varepsilon}^{b_\varepsilon},\lambda)$ denotes the number of eigenvalues of the operator $\mathcal L_{\omega_\varepsilon}^{b_\varepsilon}$ below $\lambda$, counting multiplicities.

For $0<\tilde K\varepsilon<1$, we have $(1-\tilde K\varepsilon)^{-1}\leq 1+2\tilde K\varepsilon$ and consequently,
$$\delta_\varepsilon(1-\tilde K\varepsilon)^{-1} \Big(\frac{a^2}{12}+\left(\frac{\pi}{L}\mathfrak i_0(c,a,\varepsilon)\right)^2 +M\varepsilon\Big)\leq \delta_\varepsilon(1+2\tilde K\varepsilon)\Big(\frac{a^2}{12}+\left(\frac{\pi}{L}\mathfrak i_0(c,a,\varepsilon)\right)^2 +M\varepsilon\Big)\,.$$
Thus, 
 there exists $K_1>0$ such that, for all $\varepsilon\in(0,1/\tilde K)$,
 \begin{align*}\delta_\varepsilon(1-\tilde K\varepsilon)^{-1} \Big(\frac{a^2}{12}+\left(\frac{\pi}{L}\mathfrak i_0(c,a,\varepsilon)\right)^2 +M\varepsilon\Big)&\leq \delta_\varepsilon\Big(\frac{a^2}{12}+\left(\frac{\pi}{L}\mathfrak i_0(c,a,\varepsilon)\right)^2 +K_1\varepsilon\Big)\\
 &=\delta_\varepsilon\Big(  \inf_{\ell \in\Z}m(\alpha_\ell,a)+K_1\varepsilon\Big)\,,
 \end{align*}
where $m(\alpha_\ell,a)$  is introduced in  \eqref{eq:m(alpha,delta)}.

Furthermore, by  \eqref{eq:mu-2-d-sum},  for all $k\geq 2$,
$$\delta_\varepsilon\Big(  \inf_{\ell \in\Z}m(\alpha_\ell,a)+K_1\varepsilon\Big)<\mu_k(\mathfrak H_{n,\varepsilon})\,.$$
Thus, we infer from \eqref{eq:s-gs-details},
\begin{multline}\label{eq:N-ev}
N\left(\mathcal L_{\omega_\varepsilon}^{b_\varepsilon},\frac{a^2}{12}+\left(\frac{\pi}{L}\mathfrak i_0(c,a,\varepsilon)\right)^2+M\varepsilon\right)\\
\leq {\rm Card}\left(\Big\{n\in\mathbb Z~:~\mu_1(\mathfrak H_{n,\varepsilon})\leq \delta_\varepsilon \Big(\inf_{\ell \in\Z}m(\alpha_\ell,a) +K_1\varepsilon\Big)\Big\}\right)\,.
\end{multline} 
 The condition of separation ensures that there exist a unique $n_0\in\Z$ minimizing the problem in \eqref{eq:betan} and  $d_0 >0$ such that, for all $n\in\Z\setminus\{n_0\}$
 and $\varepsilon\in (0,\varepsilon_0]$ satisfying (SC)$_\delta$, 
\begin{equation}\label{eq:beta-n>>1}
m(\alpha_n,a)\geq m(\alpha_{n_0},a)+ d_0 \,.
\end{equation}
Using \eqref{eq:lb-alpha-big}, we can restrict to  counting the set of $n\in\Z$ satisfying the conditions  $$|\alpha_n|\leq 10(a+\frac{\pi}{2L}) \mbox{ and } \mu_1(\mathfrak H_{n,\varepsilon})\leq \delta_\varepsilon \big(m(\alpha_{n_0},a) +K_1\varepsilon\big)\,.$$
For $n\in\Z\setminus\{n_0\}$ and $|\alpha_n|\leq 10(a+\frac\pi{2L})$, we know, thanks to  \eqref{eq:mu-delta=0}, that $$\mu_1(\mathfrak H_{n,\varepsilon})=\delta_\varepsilon m(\alpha_n,a)+o(\delta_\varepsilon)\,.$$  We infer from the condition in \eqref{eq:beta-n>>1}, that, for $\varepsilon$ sufficiently small,
 $$\mu_1(\mathfrak H_{n,\varepsilon})\geq \delta_\varepsilon \left(m(\alpha_{n_0},a)+\frac{d_0}2\right)\,.
 $$
  Consequently, for $\varepsilon$ sufficiently small,
$$N\left(\mathcal L_{\omega_\varepsilon}^{b_\varepsilon},\frac{a^2}{12}+\left(\frac{\pi}{L}\mathfrak i_0(c,a,\varepsilon)\right)^2+M\varepsilon\right)\leq 1\,,$$
which, when combined with \eqref{eq:ub-lambda-q}, yields the simplicity of the eigenvalue $\lambda(\varepsilon,b_\varepsilon)$.\end{proof}\medskip

\begin{rem}\label{rem:spec-gap}
Collecting \eqref{eq:d-sum} and \eqref{eq:beta-n>>1}, we get under the assumptions of Proposition~\ref{PropA} that the spectral gap between the first and second eigenvalues of $\mathcal L_{\omega_\varepsilon}^{b_\varepsilon}$ satisfies for $\varepsilon$ sufficiently small,
\begin{equation}\label{eq:sp-gap} \lambda_2(\varepsilon,b_\varepsilon)-\lambda(\varepsilon,b_\varepsilon)\geq \frac{d_0}2\,.
\end{equation}
\end{rem}

\begin{proof}[\bf Proof of Proposition \ref{PropB}]~

The problem in \eqref{eq:betan} may have  at most  two minimizers. Let $n_0$ be the {\it smallest} minimizer of  \eqref{eq:betan}.  There exist $d'>0$ and $\varepsilon_0>0$, such that for $n\in\Z\setminus\{n_0,n_0+1\}$ and $\varepsilon \in (0,\varepsilon_0]$, we have
\begin{equation}\label{eq:beta-n>>1*}
m(\alpha_n,a)\geq m(\alpha_{n_0},a)+d'\,.
\end{equation}
 Consequently,   \eqref{eq:N-ev} yields the existence of $\varepsilon_0 >0$  such that,  for  $\varepsilon\in (0,\varepsilon_0]$, 
$$N\left(\mathcal L_{\omega_\varepsilon}^{b_\varepsilon},\frac{a^2}{12}+\left(\frac{\pi}{L}\mathfrak i_0(c,a,\varepsilon)\right)^2+M\varepsilon\right)\leq 2\,,$$
where $M$ is the constant in \eqref{eq:ub-lambda-q}. 
\end{proof}
  
\begin{rem}\label{rem:PropB}
Assume that $0<\varepsilon\leq \varepsilon_0$ and the problem \eqref{eq:betan} has two minimizers $n_0$ and $m_0=n_0+1$.  Then, there exists $M'>0$ 
 and a possibly smaller $\varepsilon_0>0$ such that, for $\varepsilon \in (0,\varepsilon_0]$, the second min-max eigenvalue satisfies,
\begin{equation}\label{eq:ub-lambda-q*}
\lambda_2(\varepsilon,b_\varepsilon)\leq \frac{a^2}{12}+\left(\frac{\pi}{L}\mathfrak i_0(c,a,\varepsilon)\right)^2+M'\varepsilon\,.
\end{equation} 
This can be achieved by using the min-max formula with the two dimensional eigenspace \break 
$V_0:={\rm span}\big(v_{n_0},v_{m_0}\big)\,,$ 
where, for every integer $n$, the function $v_n$ is defined as follows,
\begin{equation}\label{eq:ef-lambda(B)}
v_{n}(s,t) = \exp\left(-i\frac{n\pi}{L}s\right) u_{n}\left(b_\varepsilon^{-1}t\right)\,,
\end{equation}
with $u_{n}$  the normalized ground state  of the effective operator $\mathfrak H_{n,\varepsilon}\,$.

This case covers the sequence $\left(\varepsilon_n\right)_{n\geq 1}$ with $\varepsilon_n= a \left(\frac1{2\gamma_0}\left(\frac{\pi n}{L}-a \right)-c\right)^{-1}$, where,  by Theorem~\ref{thm:nt-r}\,, 
$$ \lambda(\varepsilon_n, b_{\varepsilon_n})\underset{n\to+\infty}{\sim} \frac{a^2}{12}+\frac{\pi^2}{4L^2}\,.$$
An interesting question would be to determine the gap 
$\lambda_2(\varepsilon_n, b_{\varepsilon_n})-\lambda(\varepsilon_n, b_{\varepsilon_n}) \,.$
\end{rem}  

\subsection{Structure of  ground states}

When the separation condition {\rm (SC)}$_\delta$ holds, the eigenvalue $\lambda(\varepsilon,b_\varepsilon)$ is simple. We can prove that the ground states of the operator $\mathcal L_{\omega_\varepsilon}^{b_\varepsilon}$ have a simple structure.  We denote by $\Pi_\varepsilon$ the orthogonal projection  on the space of ground states of $\mathcal L_{\omega_\varepsilon}^{b_\varepsilon}$ and will have:
\begin{proposition}\label{PropC}
 For any $\delta\in(0,\frac12)$, there exists $\varepsilon_0,M_0>0$ such that, for all $\varepsilon\in(0,\varepsilon_0]$ satisfying condition  {\rm (SC)}$_\delta$, we have    
 $$ \| u_0 - \Pi_\varepsilon u_0 \|_{2,\Omega_\varepsilon}\leq M_0\, \varepsilon\,,$$
 where
 \begin{itemize}
 \item $b_\varepsilon$ is given by \eqref{eq:Be},
 \item
$\displaystyle u_0(x)= \exp \Big(-ib_\varepsilon\varphi_0(s,t)\Big)\,\exp\left(\frac{in_0\pi s}{L}\right)\,,\quad(x=\Phi_0(s,t)) \,,$
\item  $n_0\in\Z$ is the minimizer of \eqref{eq:betan0},
\item  $\varphi_0$ is the function in \eqref{eq:v},
 and $\Phi_0$ is the  diffeomorphism introduced in \eqref{eq:Phi0}.
 \end{itemize}
\end{proposition}

Before the proof we recall an abstract lemma in Hilbertian analysis which reads in our application as follows:

\begin{lem}\label{lemC}
Assume that $\varepsilon\in(0,1)$, $\mathcal K>0$  and $v\in H^1(\Omega_\varepsilon)$  satisfy 
\begin{equation}\label{eq:qfC}
Q_{\varepsilon,b_\varepsilon}(v):=\int_{\Omega_\varepsilon}|(\nabla-ib_\varepsilon \Fb)v|^2\,dx\leq  \lambda(\varepsilon,b_\varepsilon)\|v\|_{2,\Omega_\varepsilon}^2+\mathcal K\,.
\end{equation}
Then
\begin{equation}\label{asslema}
Q_{\varepsilon,b_\varepsilon}(v-\Pi_\varepsilon v)\leq \lambda(\varepsilon,b_\varepsilon)\|v-\Pi_\varepsilon v\|_{2,\Omega_\varepsilon}^2+\mathcal K\,,
\end{equation}
and
\begin{equation}\label{asslemb} \Big(\lambda_2(\varepsilon,b_\varepsilon)-\lambda(\varepsilon,b_\varepsilon\Big)\|v-\Pi_\varepsilon v\|^2_{2,\Omega_\varepsilon}\leq 
\mathcal K\,.
\end{equation}
Here $b_\varepsilon$ is given in \eqref{eq:Be}.
\end{lem}
 We will use Lemma~\ref{lemC} in the proof of Proposition~\ref{PropC} and also later in Section \ref{Sec:min*}. For the  convenience of the reader, we recall its standard proof.

\begin{proof}[Proof of Lemma~\ref{lemC}]~\\
We start by observing the following  two  identities
$$\|v\|_{2,\Omega_\varepsilon}^2=\|v-\Pi_\varepsilon v\|_{2,\Omega_\varepsilon}^2+\|\Pi_\varepsilon v\|_{2,\Omega_\varepsilon}^2\,,$$
and
\begin{align*}
 Q_{\varepsilon,b_\varepsilon}(v)&=
  Q_{\varepsilon,b_\varepsilon}(v-\Pi_\varepsilon v)+ Q_{\varepsilon,b_\varepsilon}(\Pi_\varepsilon v)\\
  &=Q_{\varepsilon,b_\varepsilon}(v-\Pi_\varepsilon v)+\lambda(\varepsilon,b_\varepsilon)\|\Pi_\varepsilon v\|_{2,\Omega_\varepsilon}^2\,.
  \end{align*}
  This implies through  \eqref{eq:qfC} the inequality 
  \eqref{asslema}.\\
  
Now, we write by the min-max principle,
 $$Q_{\varepsilon,b_\varepsilon}(v-\Pi_\varepsilon v)\geq \lambda_2(\varepsilon,b_\varepsilon)\|v-\Pi_\varepsilon v\|_{2,\Omega_\varepsilon}^2\,.$$
Collecting the foregoing estimates and \eqref{eq:qfC}, we get
$$\lambda_2(\varepsilon,b_\varepsilon)\|v-\Pi_\varepsilon v\|_{2,\Omega_\varepsilon}^2+\lambda(\varepsilon,b_\varepsilon)\|\Pi_\varepsilon v\|_{2,\Omega_\varepsilon}^2 \leq \lambda(\varepsilon,b_\varepsilon)\Big(\|v-\Pi_\varepsilon v\|_{2,\Omega_\varepsilon}^2+\|\Pi_\varepsilon v\|_{2,\Omega_\varepsilon}^2\Big)+\mathcal K\,,$$
which  gives \eqref{asslemb} and finishes the proof of Lemma~\ref{lemC}.
\end{proof}
\medskip
\begin{proof}[Proof of Proposition~\ref{PropC}]~\\
Let $\mathfrak e_0(c,a,\varepsilon)$ be the quantity  introduced in \eqref{eq:eff-en*}.  It is easy to check that
$$\int_{\Omega_\varepsilon}|u_0|^2\,dx=2L \varepsilon+\mathcal O(\varepsilon^2)\,,$$
and
$$
Q_{\varepsilon,b_\varepsilon}(u_0):=\int_{\Omega_\varepsilon}|(\nabla-ib_\varepsilon \Fb)u_0|^2\,dx\leq  2L  \varepsilon\, \mathfrak e_0(c,a,\varepsilon)+\mathcal O(\varepsilon^2)\,.$$
Now, using Theorem~\ref{thm:nt-r}, we may write
\begin{equation}\label{eq:qf-u0*}
Q_{\varepsilon,b_\varepsilon}(u_0)\leq \lambda(\varepsilon,b_\varepsilon)\|u_0\|_{2,\Omega_\varepsilon}^2+\mathcal O(\varepsilon^2)\,.
\end{equation}
By Lemma~\ref{lemC}, we deduce that
$$\Big(\lambda_2(\varepsilon,b_\varepsilon)-\lambda(\varepsilon,b_\varepsilon)\Big)\|u_0-\Pi_\varepsilon u_0\|^2_{2,\Omega_\varepsilon} = \mathcal O(\varepsilon^2)\,.$$
To finish the proof, we use the lower bound  of the spectral gap given in Remark~\ref{rem:spec-gap}.
\end{proof}

\begin{rem}\label{remC}
Proposition~\ref{PropC} yields the existence of $\tilde M_0$ such that,  for all $\varepsilon\in(0,\varepsilon_0]$ and   $u\in L^2(\Omega_\varepsilon)$,
$$\left\| \Pi_\varepsilon u- \frac1{|\Omega_\varepsilon|} \langle u,u_0\rangle u_0\right\|_{2,\Omega_\varepsilon}\leq \tilde M_0 \,  \varepsilon^{1/2} \, \|u\|_{2,\Omega_\varepsilon}\,.$$

Indeed, since the eigenvalue $\lambda(\varepsilon,b_\varepsilon)$
is simple, the corresponding eigenspace is spanned by the following normalized ground state
$$u_\varepsilon=\frac1{\|\Pi_\varepsilon u_0\|_{2,\Omega_\varepsilon}}\Pi_\varepsilon u_0\,,$$
and
$$\Pi_\varepsilon u=\langle u,u_\varepsilon\rangle u_\varepsilon\,.$$
\end{rem}
\medskip
\subsection{Breakdown of superconductivity}~
A celebrated result   by Giorgi-Phillips   \cite{GP} establishes the breakdown of superconductivity when the parameter measuring the strength of the magnetic field is sufficiently large. One consequence of the main results of this paper is the following `quantitative' version of the breakdown of superconductivity.

\begin{proposition}
Given $\kappa>0$ and $a>  2\sqrt{3}\kappa$,  there exists $\varepsilon_0>0$ such that, for all $\varepsilon\in(0,\varepsilon_0]$,  all $H\geq \frac{a}{\varepsilon}$, 
 every critical point $(\psi,\Ab)_{\kappa,H,\varepsilon}$ is trivial.
 \end{proposition}
\begin{proof}
We first  see from  Theorem~\ref{rem:t-r} (assertion (A), with $N >\kappa^2$) together with Theorem~\ref{thm:hc3},
 that this is true for $H\geq \frac{d_0}{\varepsilon}$.
 
 Assuming now that $H\leq \frac{d_0}{\varepsilon}$, we prove it by contradiction.
 If there were  sequences $(H_n)_{n\geq 1}$ and $(\varepsilon_n)_{n\geq 1}$ such that $\varepsilon_n\to0$, $ H_n \, \varepsilon_n\to\alpha$ for some  $\alpha\in[a,+\infty)$, and a non trivial minimizer, then an easy adjustment of the proof of Theorem~\ref{thm:nt-r} yields that
$$\lambda(\varepsilon_n,H_n) \sim\frac{\alpha^2}{12}+\left(\frac{\pi}L\mathfrak i_0(0,\alpha,\varepsilon_n)\right)^2\,.$$
 Consequently, we get $\lambda(\varepsilon_n,H_n)\geq \kappa^2$ for $n$ large enough, because $\frac{\alpha^2}{12}>\kappa^2$. Theorem~\ref{thm:hc3} leads to a contradiction.
\end{proof}
\subsection{Lack of strong diamagnetism and oscillations in the Little-Parks framework}\label{rem:little-parks}~ 

The behavior of the eigenvalue  in Theorem~\ref{thm:nt-r} shows a pleasant connection to the oscillatory behavior of the Little-Parks experiment. 
The following statement displays counterexamples to strong diamagnetism.

\begin{proposition}\label{prop:little-parks}
There exists a sequence $(\varepsilon_N)_{N\geq 1}\subset\R_+$ which converges to $0$ such that, for all $N\geq 1$,  the function $H\mapsto\lambda(\varepsilon_N,H)$ is not monotone increasing.  
\end{proposition}
\begin{proof}
Choose $a>0$ so that
$$\frac{a^2}{12}< \kappa^2<\frac{a^2}{12}+\frac14\left(\frac{\pi}{L}\right)^2\,.$$
  Let us define the following sequence 
$$\tilde \varepsilon_N=\frac{a L\gamma_0}{\pi}\left(N-\frac{a}{2}\right)^{-1}\quad\text{for }  N \in \mathbb N \cap (\frac{a}{2},+\infty)\,.$$
Define $H_N^{(1)}<H_N^{(2)}<H_N^{(3)}$ by
$$H_{N}^{(1)}=\frac{a}{\tilde \varepsilon_N}-\frac{\pi}{2\gamma_0L}\,,\quad
H_{N}^{(2)}=\frac{a}{\tilde \varepsilon_N}\,,\quad H_{N}^{(3)}=\frac{a}{\tilde \varepsilon_N}+\frac{\pi}{2\gamma_0L}\,.$$
Then, we notice that, as $N\to+\infty$,
$$\lambda(\tilde \varepsilon_N,H_N^{(1)})\sim \frac{a^2}{12}+\frac14\left(\frac{\pi}{L}\right)^2>\kappa^2\,,\quad \lambda(\tilde \varepsilon_N,H_N^{(2)})\sim \frac{a^2}{12}<\kappa^2\,,\quad\lambda(\tilde \varepsilon_N,H_N^{(3)})\sim \frac{a^2}{12}+\frac14\left(\frac{\pi}{L}\right)^2>\kappa^2\,.$$
Hence we find $N_0$ such that the statement of the proposition holds for $\varepsilon_N= \tilde \varepsilon_{N+N_0}$.
\end{proof}

\begin{rem}\label{rem:little-parks*}
Along the proof of Proposition~\ref{prop:little-parks}, we obtain the two remarkable observations:
\begin{itemize}
\item For $N$ sufficiently large $H_N^{(1)}<H_N^{(2)}$ while
$\lambda(\tilde\varepsilon_N, H_N^{(1)})>\lambda(\tilde\varepsilon_N, H_N^{(2)})$.\medskip
\item 
By Theorem~\ref{thm:hc3}, for large  $N$,  the minimizers $(\psi,\Ab)_{\kappa,H_N^{(i)},\tilde\varepsilon_N}$\,, $i=1,3$, are non-trivial, while any critical point $(\psi,\Ab)_{\kappa,H_N^{(2)},\tilde\varepsilon_N}$ is trivial. 
\end{itemize}
Thus,  the transition from the superconducting to the normal state is not monotone, which is in agreement with the Little-Parks experiment.
\end{rem}

\section{Structure of the order parameter and circulation of the super-current} \label{Sec:min*}

\subsection{Hypotheses}\label{sec:hyp}

Throughout this section, we work under the following hypothesis on the parameter $H$:
\begin{equation}\label{eq:assH}
H=b_\varepsilon:=\frac{a}\varepsilon+c
\end{equation}
where $a>0$ and $c\in\R$ are fixed constants.

The results of this section will concern an arbitrary minimizer $(\psi,\Ab)_{\varepsilon,H}$  of the GL functional, provided $H$ satisfies \eqref{eq:assH},  and $\varepsilon$ satisfies the  `separation' conditions (SC)$_\delta$ and (SC)$'_\delta$ introduced in \eqref{eq:ass-sep}-\eqref{eq:ass-sep'}.

\subsection{Approximation of the order parameter}~\\
In light of Theorem~\ref{thm:min}, we introduce the following quantity
\begin{equation}\label{eq:Lk}
\Lambda_\kappa^\varepsilon:=\frac{\kappa^2-\mathfrak{e}_0(c,a,\varepsilon)}{\kappa^2}\,,
\end{equation}
where $\mathfrak{e}_0(c,a,\varepsilon)$ is introduced in \eqref{eq:eff-en*}.
Note that, under the hypotheses in Subsection~\ref{sec:hyp}, there exists a constant $c_0>0$ such that, for all $\varepsilon$ sufficiently small,
\begin{equation}\label{eq:L-k>0}
c_0\leq \Lambda_\kappa^\varepsilon\leq 1\,.
\end{equation}
Let $u_0$ be the function introduced in Proposition~\ref{PropC}. We will prove that, up to multiplication by $\sqrt{\Lambda_\kappa^\varepsilon}$ and a complex phase, the function $u_0$ provides us with a good approximation of the GL order parameter $\psi$. 

\begin{prop}\label{prop:app-psi}
There exist constants $C,\varepsilon_0>0$ such that,  if
\begin{itemize}
\item $\varepsilon\in(0,\varepsilon_0]$ satisfies the separation conditions {\rm (SC)}$_\delta$ and {\rm (SC)}$'_\delta$\,;
\item  $H$ satisfies \eqref{eq:assH}\,;
\item $(\psi,\Ab)_{\varepsilon,H}$ is a minimizer of the GL functional in \eqref{eq:GL*-omega}\,;
\end{itemize}
then, there exists $\alpha_\varepsilon\in\C$ such that $|\alpha_\varepsilon|=1$, $\psi$ satisfies 
\begin{equation}\label{eq:app-psi3}
\big\|\psi-\alpha_\varepsilon\sqrt{\Lambda_\kappa^\varepsilon}\, u_0\big\|_{2,\Omega_\varepsilon}\leq C \,  \varepsilon\,,
\end{equation}
and its trace on $\partial \Omega$ satisfies
\begin{equation}\label{eq:app-psi-bnd2}
\big\|\psi_{/\partial \Omega} -\alpha_\varepsilon\sqrt{\Lambda_\kappa^\varepsilon}\, ( u_0)_{/\partial \Omega} \big\|_{L^2(\partial\Omega)}
\leq C \,  \varepsilon^{1/2}\,.
\end{equation}
\end{prop}
\begin{proof}~
\begin{proof}[Proof of \eqref{eq:app-psi3}]
Collecting \eqref{eq:GL-lb**} and \eqref{uppb6}, we infer from Theorem~\ref{thm:min},
\begin{equation}\label{eq:enEF}
\mathcal E_\varepsilon(\psi,\Fb)=  - \frac12(\Lambda_\kappa^\varepsilon)^2|\Omega_\varepsilon|+\mathcal O(\varepsilon^2)\,.
\end{equation}
Furthermore, it results from Theorem~\ref{thm:min} (see \eqref{eq:1.23}) together with the definiton of $\Lambda_\kappa^\varepsilon$ in \eqref{eq:Lk}  that
\begin{equation}\label{eq:conv-L2,4}
\Big\|\, |\psi|^2-\Lambda_\kappa^\varepsilon\,\Big\|_{2,\Omega_\varepsilon}=\mathcal O(\varepsilon) \mbox{ and }
\Big\|\, |\psi|-\sqrt{\Lambda_\kappa^\varepsilon}\,\Big\|_{2,\Omega_\varepsilon}=\mathcal O(\varepsilon)\,. \end{equation}
Consequently, 
\begin{equation}\label{eq:norm-psi}
\|\psi\|_{4,\Omega_\varepsilon}^4=(\Lambda_\kappa^\varepsilon)^2|\Omega_\varepsilon|+\mathcal O(\varepsilon^2)~\quad{\rm and}~\quad\|\psi\|_{2,\Omega_\varepsilon}^2=\Lambda_\kappa^\varepsilon|\Omega_\varepsilon|+\mathcal O(\varepsilon^2)\,.
\end{equation}
Note that \eqref{eq:Lk} yields that $\kappa^2(-(\Lambda_\kappa^\varepsilon)^2+\Lambda_\kappa^\varepsilon)=\Lambda_\kappa^\varepsilon\mathfrak e_0(c,a,\varepsilon)$\,, which in turn yields the following  identity,
\begin{align*}
-\frac{\kappa^2}{2} (\Lambda_\kappa^\varepsilon)^2|\Omega_\varepsilon|+\kappa^2\|\psi\|_{2,\Omega_\varepsilon}^2-\frac{\kappa^2}{2}\|\psi\|_{4,\Omega_\varepsilon}^4&=
|\Omega_\varepsilon|\Lambda_\kappa^\varepsilon\, \mathfrak e_0(c,a,\varepsilon)+\mathcal O(\varepsilon^2)\\
&=\mathfrak e_0(c,a,\varepsilon) \|\psi\|_{2,\Omega_\varepsilon}^2+\mathcal O(\varepsilon^2)\,.
\end{align*}
Now we insert this identity into \eqref{eq:enEF} to  get (see \eqref{eq:GL*-omega} and \eqref{eq:lambda-GL-ep}):
$$
\int_{\Omega_\varepsilon}|(\nabla-iH\Fb)\psi|^2\,dx=\mathfrak e_0(c,a,\varepsilon) \|\psi\|_{2,\Omega_\varepsilon}^2+\mathcal O(\varepsilon^2)\,.
$$
Recall that $H=b_\varepsilon$ with $b_\varepsilon$  given in \eqref{eq:assH}. Using Theorem~\ref{thm:nt-r} and the definition of $\mathfrak e_0(c,a,\varepsilon)$ in \eqref{eq:eff-en*}, we get further
\begin{align*}
\int_{\Omega_\varepsilon}|(\nabla-ib_\varepsilon\Fb)\psi|^2\,dx&=\left(\lambda(b_\varepsilon,\varepsilon)+\mathcal O(\varepsilon)\right)\|\psi\|_{2,\Omega_\varepsilon}^2+\mathcal O(\varepsilon^2)\\
&\leq  \lambda(b_\varepsilon,\varepsilon)\|\psi\|_{2,\Omega_\varepsilon}^2+\mathcal O(\varepsilon^2)\,.
\end{align*}
 Now, we can apply  Lemma~\ref{lemC} (with  $v=\psi$). Using the estimate in  \eqref{eq:sp-gap}, we get
\begin{equation}\label{eq:app-psi1}
\|\psi-\Pi_\varepsilon\psi\|_{2,\Omega_\varepsilon}=\mathcal O(\varepsilon)\,,
\end{equation}
and 
\begin{equation}\label{eq:app-psi1*}
\big\|(\nabla-iH\Fb)(\psi-\Pi_\varepsilon\psi)\big\|_{2,\Omega_\varepsilon}=\mathcal O(\varepsilon)\,.
\end{equation}
Let $u_0$ be the function introduced in Proposition~\ref{PropC}. By  Remark~\ref{remC}, we know that
\begin{equation}\label{eq:app-psi2}
\left\|\Pi_\varepsilon\psi- \frac1{|\Omega_\varepsilon|} \langle \psi,u_0\rangle u_0\right\|_{2,\Omega_\varepsilon}=  \mathcal O(\varepsilon^{1/2})\|\psi\|_{2,\Omega_\varepsilon}=\mathcal O(\varepsilon)\,.
\end{equation}
We can estimate $\langle \psi,u_0\rangle$ as follows.\\  On  one hand we have
\begin{align*}
\langle \Pi_\varepsilon\psi,\psi\rangle 
&=\|\Pi_\varepsilon\psi\|^2_{2,\Omega_\varepsilon}\\
&=\|\psi\|_{2,\Omega_\varepsilon}^2-\|\psi-\Pi_\varepsilon\psi\|_{2,\Omega_\varepsilon}^2=\Lambda_\kappa^\varepsilon |\Omega_\varepsilon|+\mathcal O(\varepsilon^2)\,,
\end{align*}
by \eqref{eq:norm-psi} and \eqref{eq:app-psi1}.\\
On the other hand, using \eqref{eq:app-psi1} and  \eqref{eq:app-psi2}, we have
$$\langle \Pi_\varepsilon\psi,\psi\rangle = \frac1{|\Omega_\varepsilon|} |\langle \psi,u_0\rangle|^2+ \mathcal O(\varepsilon)\|\psi\|_{2,\Omega_\varepsilon}=\frac1{|\Omega_\varepsilon|}|\langle \psi,u_0\rangle|^2+ \mathcal O(\varepsilon^{3/2}) \,,$$
thereby obtaining that
$$|\langle \psi,u_0\rangle|^2=|\Omega_\varepsilon|\Big(\Lambda_\kappa^\varepsilon|\Omega_\varepsilon| +\mathcal O(\varepsilon^{3/2})\Big)=|\Omega_\varepsilon|^2\Lambda_\kappa^\varepsilon + \mathcal O(\varepsilon^{5/2})\,.$$
Now, we set
\begin{equation}\label{eq:alpha-e}
 \alpha_\varepsilon:=\frac{\langle\psi,u_0\rangle }{|\langle\psi,u_0\rangle|}\,,
\end{equation}
We observe that
\begin{equation}\label{eq:alpha-e*}
\left|\alpha_\varepsilon\sqrt{\Lambda_\kappa^\varepsilon}\,- \frac1{|\Omega_\varepsilon|}\langle\psi,u_0\rangle\right|=  \mathcal O(\varepsilon^{1/2})\,,
\end{equation}
and, after collecting \eqref{eq:app-psi1} and \eqref{eq:app-psi2},
\begin{equation}\label{eq:app-psi3*}
\big\|\psi-\alpha_\varepsilon\sqrt{\Lambda_\kappa^\varepsilon}\, u_0\big\|_{2,\Omega_\varepsilon}=\mathcal O(\varepsilon)\,.
\end{equation}
\end{proof}
\begin{proof}[Proof of \eqref{eq:app-psi-bnd2}]
We first compute,
\begin{multline*}
\big\|(\nabla-iH\Fb)(\Pi_\varepsilon\psi-\alpha_\varepsilon\sqrt{\Lambda_\kappa^\varepsilon}\,u_0)\big\|_{2,\Omega_\varepsilon}^2=\big\|(\nabla-iH\Fb)\Pi_\varepsilon\psi\|_{2,\Omega_\varepsilon}^2+\big\|(\nabla-iH\Fb)\alpha_\varepsilon\sqrt{\Lambda_\kappa^\varepsilon}\,u_0\big\|_{2,\Omega_\varepsilon}^2\\
+2{\rm Re}\langle (\nabla-iH\Fb)\Pi_\varepsilon\psi,  (\nabla-iH\Fb)\alpha_\varepsilon\sqrt{\Lambda_\kappa^\varepsilon}u_0\rangle\,.
\end{multline*}
We perform an integration by parts to rewrite the last term of above in the form
\begin{multline*}\langle (\nabla-iH\Fb)\Pi_\varepsilon\psi,  (\nabla-iH\Fb)\alpha_\varepsilon\sqrt{\Lambda_\kappa^\varepsilon}u_0\rangle=\langle -(\nabla-iH\Fb)^2\Pi_\varepsilon\psi,  \alpha_\varepsilon\sqrt{\Lambda_\kappa^\varepsilon}u_0\rangle\\=\lambda(\varepsilon,H)\langle \Pi_\varepsilon\psi,  \alpha_\varepsilon\sqrt{\Lambda_\kappa^\varepsilon}u_0\rangle\,.
\end{multline*}
We then insert   \eqref{eq:qf-u0*}  and get,
\begin{align*}
\big\|(\nabla-iH\Fb)(\Pi_\varepsilon\psi-\alpha_\varepsilon\sqrt{\Lambda_\kappa^\varepsilon}\,u_0)\big\|_{2,\Omega_\varepsilon}^2&\leq \lambda(\varepsilon,H)\|\Pi_\varepsilon\psi-\alpha_\varepsilon\sqrt{\Lambda_\kappa^\varepsilon}\,u_0)\big\|_{2,\Omega_\varepsilon}^2+\mathcal O(\varepsilon^2)\\
&=\mathcal O(\varepsilon^2)\,,
\end{align*}
where we used \eqref{eq:app-psi2} and \eqref{eq:alpha-e*} for  the last statement above. Combining this estimate and \eqref{eq:app-psi1*}, we get
\begin{equation}\label{eq:app-psi4}
\big\|(\nabla-iH\Fb)(\psi-\alpha_\varepsilon\sqrt{\Lambda_\kappa^\varepsilon}\,u_0)\big\|_{2,\Omega_\varepsilon}=\mathcal O(\varepsilon)\,.
\end{equation}
Let us introduce the function
$$w=\left|\psi-\alpha_\varepsilon\sqrt{\Lambda_\kappa^\varepsilon}\,u_0\right|\,.$$
By the diamagnetic inequality, we infer from \eqref{eq:app-psi4},
\begin{equation}\label{eq:app-psi-bd1}
\|\nabla  w\|_{2,\Omega_\varepsilon}=\mathcal O(\varepsilon) \,.
\end{equation}
Define now the re-scaled function
$$[-L,L)\times(0,\varepsilon_0)\ni (s,\tau)\mapsto \tilde w_\varepsilon(s,\tau)=\tilde w (s,\varepsilon_0^{-1}\varepsilon \tau)\,,
$$
where $\tilde w=w\circ\Phi_0$, $\Phi_0$ is the transformation introduced in \eqref{eq:Phi0}, and $\varepsilon_0\in(0,1)$ is a sufficiently small constant so that the transformation $\Phi_0:[-L,L)\times(0,\varepsilon_0)\to \Omega_{\varepsilon_0}$ is bijective.

We can define a function $w_\varepsilon\in H^1(\Omega_{\varepsilon_0})$ by means of the function $\tilde w_\varepsilon$ as follows
$$w_\varepsilon(x)=\tilde w_\varepsilon(s,\tau)\quad{\rm for~}  x=\Phi_0(s,\tau)\,.$$
Consequently, we obtain from \eqref{eq:app-psi3*} and  \eqref{eq:app-psi-bd1},
$$\|w_\varepsilon\|_{H^1(\Omega_{\varepsilon_0})}=\mathcal  O(\varepsilon^{1/2})\,.$$
By the trace theorem, we deduce that 
$$\|w_\varepsilon\|_{L^2(\partial\Omega)}=\mathcal O(\varepsilon^{1/2})\,.$$
\end{proof}
Having proved \eqref{eq:app-psi3} and \eqref{eq:app-psi-bnd2}, we have achieved  the proof of Proposition \ref{prop:app-psi}.
\end{proof}
\medskip

\subsection{More a priori estimates}~\\

Using the curl-div estimate, we can write,
\begin{equation}\label{eq:A-F:H1}
\|\Ab-\Fb\|_{H^1(\Omega)}\leq C\, \|\curl(\Ab-\Fb)\|_{L^2(\Omega)}=\mathcal O(\varepsilon^3)\,,
\end{equation}
where we used \eqref{eq:apriori} and \eqref{eq:A-F1} to get the estimate $\mathcal O(\varepsilon^3)$.

Also, the following estimate holds (see \cite[Lem.~B.1]{AK})
\begin{equation}\label{eq:A-F:H2}
\|\A-\Fb\|_{H^2(\Omega)}\leq C\, \|\nabla\curl(\Ab-\Fb)\|_{L^2(\Omega)}=\mathcal O(\varepsilon^2)\,,
\end{equation}
where we used \eqref{eq:apriori} to  get the estimate $\mathcal O(\varepsilon^2)$.

Consequently, the Sobolev embedding theorem yields, for every $\alpha\in(0,1)$,
\begin{equation}\label{eq:A-F:C-holder}
\|\A-\Fb\|_{C^{0,\alpha}(\overline\Omega)}=\mathcal O(\varepsilon^2)\,.
\end{equation}
~

 \subsection{Proof of Theorem \ref{thm:min*}}~\\
With the following  notation
$$(a,b)={\rm Re}\big( a\bar b\big)\quad(a,b\in\C)\,,$$
we may express the super-current as follows
$$\jb =(i\psi,(\nabla-iH\Ab)\psi)\,.$$
We will prove (see \eqref{eq:th16}) that
\begin{equation}\label{eq:sc}
\frac1{|\partial\Omega|}\int_{\partial\Omega}\mathbf t\cdot \jb \,ds=\Lambda_\kappa^\varepsilon\frac{4\pi n_0}{|\partial\Omega|}+o(\varepsilon^{-1})\quad(\varepsilon\to0_+)\,,
\end{equation}
where $\mathbf t$ is the unit tangent vector of $\partial\Omega$ oriented in the counter-clockwise direction, and  $n_0\in\mathbb Z$ is the minimizer of \eqref{eq:betan}. Note that $n_0$ depends on $\varepsilon$ and is  $\mathcal O (\varepsilon^{-1})$, as $\varepsilon\to0_+$.

\begin{lem}\label{lem:sc}
$$\int_{\partial\Omega}\mathbf t\cdot \jb \,ds=-H\Lambda_\kappa^\varepsilon |\Omega|+\int_{\partial\Omega}\mathbf t\cdot (i\psi,\nabla\psi)\,ds+o(\varepsilon^{-1})\quad(\varepsilon\to0_+)\,.$$
\end{lem}
\begin{proof}
We perform the simple decomposition
$$\jb =-H\Ab |\psi|^2+(i\psi,\nabla\psi)\,.$$ 
In light of  \eqref{eq:app-psi-bnd2} and \eqref{eq:A-F:C-holder}, we write,
$$\int_{\partial\Omega} \mathbf t\cdot \Ab\,  |\psi|^2\,ds=\Lambda_\kappa^\varepsilon \int_{\partial\Omega} \mathbf t\cdot \Fb \,ds+o(1)\quad(\varepsilon\to0_+)\,.$$
By  the Stokes formula,
$$\int_{\partial\Omega}\mathbf t\cdot \Fb \,ds=\int_{\Omega}\curl\Fb\,dx=|\Omega|\,.$$
\end{proof}

Let $\Phi_0$ be the transformation introduced in \eqref{eq:Phi0}. Denote by $\tilde\psi=\psi\circ\Phi_0^{-1}$ and define the function $u=\tilde u\circ\Phi_0$ as follows
\begin{equation}\label{eq:anstaz-tpsi}  
 [-L,L)\times(0,\varepsilon)\ni(s,t)\mapsto \tilde u(s,t) := (\Lambda_\kappa^\varepsilon)^{-\frac 12}\alpha_\varepsilon^{-1} \,
e^{iH\varphi_0(s,t)}e^{-in_0\pi s/L}  \tilde \psi(s,t) \,,
\end{equation}
where $\varphi_0$ is introduced in \eqref{eq:v}  
and $\alpha_\varepsilon$ is the unit complex number defined in \eqref{eq:alpha-e}. Thanks to \eqref{eq:L-k>0},  $\tilde u$ is well defined by \eqref{eq:anstaz-tpsi} and satisfies
\begin{equation}\label{eq:app-psi-tilde-u}
|\tilde u(s,t)|\leq \frac1{\sqrt{c_0}}|\tilde\psi(s,t)|\leq \frac1{\sqrt{c_0}}\,.
\end{equation}
Furthermore, it results from  \eqref{eq:app-psi-bnd2}  that $\tilde u\big|_{t=0} $ converges  to $1$ in $L^2([-L,L))$.

\begin{lem}\label{lem:del-ts}
$$
\int_{-L}^L\int_0^\varepsilon \Big(|\partial_t\tilde u|^2+|\partial_s\tilde u|^2\Big)dtds=\mathcal O(\varepsilon)\,.
$$
\end{lem}
\begin{proof}
 By a computation analgous with the one in  \eqref{eq:qf},
$$\mathcal E_\varepsilon(\psi,\Fb)=\Lambda_\kappa^\varepsilon\big(1+\mathcal O(\varepsilon)\big)\int_{-L}^L\int_0^\varepsilon\Big(|\partial_t\tilde u|^2+|(\partial_s+iV_\varepsilon)\tilde u|^2\Big)dtds-\kappa^2\|\psi\|_{2,\Omega_\varepsilon}^2+\frac{\kappa^2}2\|\psi\|_{4,\Omega_\varepsilon}^4\,,$$
where
$$
V_\varepsilon(s,t)=\frac\pi{L}n_0+Hf(s,t)\,,
$$
and $f$ is introduced in \eqref{eq:f}.

Consequently, we infer from \eqref{eq:enEF} and  \eqref{eq:norm-psi},
$$
\int_{-L}^L\int_0^\varepsilon |\partial_t\tilde u|^2dtds=\mathcal O(\varepsilon)\quad{\rm and}\quad \int_{-L}^L\int_0^\varepsilon |(\partial_s-iV_\varepsilon)\tilde u|^2dtds=\mathcal O(\varepsilon)\,.
$$
To finish the proof, it remains to to prove that $\int_{-L}^L\int_0^\varepsilon |\partial_s\tilde u|^2\,dtds=\mathcal O(\varepsilon)$. To that end, it is enough to prove that $\int_{-L}^L\int_0^\varepsilon |V_\varepsilon\tilde u|^2\,dtds=\mathcal O(\varepsilon)$.\medskip\\
 Since $n_0$ minimizes \eqref{eq:betan},
\begin{equation}\label{eq:n0=H}
\left|n_0+\frac{L}{\pi}\left(\frac{a}{\varepsilon}+c\right)\gamma_0\right|\leq \beta_{n_0}(c,a,\varepsilon)+\frac{a}2\leq \frac12+\frac{a}2\,.\end{equation}
Since $H=\frac{a}{\varepsilon}+c$ and  $t\in(0,\varepsilon)$,
\begin{align*}
V_\varepsilon(s,t)&=\frac\pi{L}\left(n_0+\frac{\gamma_0H L}{\pi}  +\frac{H tL}{\pi}+\frac{Ht^2L}{2\pi}k(s)\right)\\
&=\frac\pi{L}\left(n_0+\frac{L}{\pi}\left(\frac{a}{\varepsilon}+c\right)\gamma_0\right)  +\mathcal O(Ht)+\mathcal O(Ht^2)\\
&=\mathcal O(1)\,.
\end{align*}
 Now, the foregoing estimate and \eqref{eq:app-psi-tilde-u} yield,
 $$\int_{-L}^L\int_0^\varepsilon | V_\varepsilon\tilde u|^2dtds=\mathcal O(1)\int_{-L}^L\int_0^\varepsilon |\tilde u|^2dtds=\mathcal O(\varepsilon)\,.$$
\end{proof}

\begin{lem}\label{lem:Jac}
$$
\frac1{|\partial\Omega|}\int_{\partial\Omega}\mathbf t\cdot(i\psi,\nabla\psi)\,ds= \Lambda_\kappa^\varepsilon \frac{4\pi n_0}{|\partial\Omega|}+o(\varepsilon^{-1})\,.
$$
\end{lem}
\begin{proof}
Notice that
$$\int_{\partial\Omega}\mathbf t\cdot(i\psi,\nabla\psi)\,ds=\int_{-L}^L (i\tilde\psi,\partial_s\tilde\psi)\big|_{t=0}\,ds\,,$$
where (see \eqref{eq:anstaz-tpsi})
$$
(i\tilde\psi,\partial_s\tilde\psi)\big|_{t=0}=\Lambda_\kappa^\varepsilon\Big( -H\partial_s\varphi_0(s,0)+n_0\frac{\pi}{L}\Big) |\tilde u(s,0)|^2+\Lambda_\kappa^\varepsilon(i\tilde u,\partial_s\tilde u)\big|_{t=0}\,.$$
Consequently,
\begin{equation}\label{eq:circ-**}
\begin{aligned}\int_{-L}^L (i\tilde\psi,\partial_s\tilde\psi)\big|_{t=0}\,ds=\,&\Lambda_\kappa^\varepsilon \frac{\pi n_0}{L}\int_{-L}^L|\tilde u(s,0)|^2ds- \Lambda_\kappa^\varepsilon H\int_{-L}^L \partial_s\varphi_0(s,0)\tilde u(s,0)|^2\,ds\\
&+\Lambda_\kappa^\varepsilon \int_{-L}^L(i\tilde u,\partial_s\tilde u)\big|_{t=0}\,ds\,,
\end{aligned}
\end{equation}
with
\begin{equation}\label{eq:int-tilde-u*-bnd}
\int_{-L}^L|\tilde u(s,0)|^2ds=2L+o(1)=|\partial\Omega|+o(1)\,,
\end{equation}
since $\tilde u\to1$ in $L^2([-L,L))$, as $\varepsilon\to0_+$.

We estimate the integral of  $\partial_s\varphi_0(s,0) |\tilde u(s,0)|^2$. By the periodicity of the function $\varphi_0$, 
$$
\int_{-L}^L  \partial_s\varphi_0(s,0) \,ds=\varphi_0(L,0)-\varphi_0(-L,0)=0\,.$$
Using \eqref{eq:app-psi-tilde-u}, 
$$\left|\int_{-L}^L \partial_s\varphi_0(s,0)(|\tilde u(s,0)|^2-1) \,ds\right|\leq
\left(1+\frac1{\sqrt{c_0}}\right)\|\partial_s\varphi_0\|_\infty\int_{-L}^L
\big|\,|\tilde u(s,0)|-1\,\big|\,ds=o(1)\,,$$
since $\tilde u\to 1$ in $L^2([-L,L))$. Thus,
\begin{equation}\label{eq:circ-u-***}
\int_{-L}^L  \partial_s\varphi_0(s,0)|\tilde u(s,0)|^2\,ds=
\int_{-L}^L  \partial_s\varphi_0(s,0)\,ds +\int_{-L}^L  \partial_s\varphi_0(s,0)(|\tilde u(s,0)|^2-1)\,ds=o(1)\,.
\end{equation}
It remains to estimate the integral of $(i\tilde u,\partial_s\tilde u)\big|_{t=0}$. In fact,
\begin{equation}\label{eq:est-sc1*}
\int_{-L}^L (i\tilde u,\partial_s\tilde u)\big|_{t=0}\,ds=\int_{-L}^L\int_0^\varepsilon \partial_t\big(\chi_\varepsilon(i\tilde u,\partial_s\tilde u)\big)\,dtds
\end{equation}
where $\chi_\varepsilon(t)$ is a cut-off function in $C_c^{\infty}([0,+\infty)$ satisfying $\chi_\varepsilon=1$ in $[0,\frac\varepsilon2)$,
${\rm sup}\chi_\varepsilon\subset[0,\varepsilon)$,  $0\leq\chi_\varepsilon\leq 1$ and $|\nabla\chi_\varepsilon|=\mathcal O(\varepsilon^{-1})$ in $[0,+\infty)$.

Note that
$$\partial_t\Big(\chi_\varepsilon(i\tilde u,\partial_s\tilde u)\Big)=(\partial_t\chi_\varepsilon)(i\tilde u,\partial_s\tilde u)+\chi_\varepsilon(i\partial_t\tilde u,\partial_su)+(i\chi_\varepsilon\tilde u,\partial_t\partial_s\tilde u)\,.$$
Using Lemma~\ref{lem:del-ts} and the Cauchy-Schwarz inequality, we get
$$\int_{-L}^L\int_0^\varepsilon |\partial_t\chi_\varepsilon| \,|(i\tilde u,\partial_s\tilde u)|\,dtds\leq \|\partial_t\chi_\varepsilon\|_{\infty,\Omega_\varepsilon}\|\tilde u\|_{2,\Omega_\varepsilon}\|\partial_s\tilde u\|_{2,\Omega_\varepsilon}=\mathcal O(1)\,, $$
and
$$\int_{-L}^L\int_0^\varepsilon|\chi_\varepsilon(i\partial_t\tilde u,\partial_s\tilde u)|\,dtds\leq \|\chi_\varepsilon\|_{\infty,\Omega_\varepsilon}\|\partial_t\tilde u\|_{2,\Omega_\varepsilon}\|\partial_s\tilde u\|_{2,\Omega_\varepsilon}=\mathcal O(\varepsilon)\,.$$
As for the term  $(i\chi_\varepsilon\tilde u,\partial_t\partial_s\tilde u)$, we do an integration by parts in the $s$-variable and use the periodicity with respect to $s$ to get
$$\int_{-L}^L\int_0^\varepsilon (i\chi_\varepsilon\tilde u,\partial_t\partial_s\tilde u)\,dtds=-\int_{-L}^L\int_0^\varepsilon (i\chi_\varepsilon\partial_s\tilde u,\partial_t\tilde u)\,dtds\,.$$
Now,  by Lemma~\ref{lem:del-ts} and the Cauchy-Schwarz inequality, 
$$\left|\int_{-L}^L\int_0^\varepsilon (i\chi_\varepsilon\tilde u,\partial_t\partial_s\tilde u)\,dtds\right|\leq \|\chi_\varepsilon\|_{\infty,\Omega_\varepsilon}\|\partial_t\tilde u\|_{2,\Omega_\varepsilon}\|\partial_s\tilde u\|_{2,\Omega_\varepsilon}=  \mathcal O(\varepsilon)\,.$$
Collecting the foregoing estimates, we infer from \eqref{eq:est-sc1*},
\begin{equation}\label{eq:circ-u-*4}
\int_{-L}^L (i\tilde u,\partial_s\tilde u)\big|_{t=0}\,ds=\mathcal O(1)\,.
\end{equation}
Inserting \eqref{eq:int-tilde-u*-bnd}, \eqref{eq:circ-u-***} and \eqref{eq:circ-u-*4} into \eqref{eq:circ-**}, we finish the proof of Lemma~\ref{lem:sc}.
\end{proof}

\begin{proof}[Proof of \eqref{eq:sc}]
By collecting the formulas in Lemmas~\ref{lem:sc} and \ref{lem:Jac}, we obtain
$$\frac1{|\partial\Omega|}\int_{\partial\Omega}\mathbf t\cdot \jb \,ds=\Lambda_\kappa^\varepsilon\left(\frac{2\pi n_0}{|\partial\Omega|}-H\gamma_0\right)+o(\varepsilon^{-1})\,. $$
This fomula yields \eqref{eq:sc} since $n_0=-\frac{\pi}{L}H\gamma_0+\mathcal O(1)$, by \eqref{eq:n0=H}. Having proved \eqref{eq:sc}, we have finished  the proof of Theorem~\ref{thm:min*}.
\end{proof}

\subsection*{Acknowledgements} This work started while the authors visited the Mittag-Leffler Institute in January 2019. A. Kachmar is supported by the Lebanese University  within the project ``Analytical and numerical aspects of the Ginzburg-Landau model''.

\v0.2in

\end{document}